\documentclass{amsart}
\usepackage{tikz}
\usepackage{amsmath,amsfonts,amsthm,amssymb,amscd, verbatim, graphicx,textcomp, hyperref}
\usepackage[margin=2cm]{geometry}

\usepackage{pinlabel}
\usepackage{lscape}
\usepackage{latexsym}
\usepackage{multirow}
\newtheorem{Thm}{Theorem}[section]
\newtheorem{Cor}[Thm]{Corollary}

\newtheorem{Prop}[Thm]{Proposition}
\newtheorem{Qu}[Thm]{Question}
\newtheorem{Lem}[Thm]{Lemma}

\newtheorem{Hyp}[Thm]{Hypotheses}

\newtheorem*{thma}{Theorem A}
\newtheorem*{thmb}{Theorem B}
\newtheorem*{thmc}{Theorem C}

\theoremstyle{definition}
\newtheorem{Def}[Thm]{Definition}

\newtheorem{Ex}[Thm]{Example}

\theoremstyle{remark}

\numberwithin{equation}{section}

\newcommand{\Aut}{\operatorname{Aut}}

\newcommand{\stab}{\operatorname{stab}}
\newcommand{\Id}{\operatorname{Id}}

\renewcommand{\dim}{\operatorname{dim}}

\newcommand{\D}{\mathcal{D}}
\newcommand{\De}{\mathcal{D}}

\newcommand{\Sym}{\operatorname{Sym}}
\newcommand{\Fi}{\operatorname{Fi}}
\newcommand{\supp}{\operatorname{supp}}
\newcommand{\Alt}{\operatorname{Alt}}
\newcommand{\PSL}{\operatorname{PSL}}

\newcommand{\PSU}{\operatorname{PSU}}
\newcommand{\GL}{\operatorname{GL}}
\newcommand{\Sp}{\operatorname{Sp}}

\renewcommand{\Gamma}{\varGamma}
\renewcommand{\epsilon}{\varepsilon}

\renewcommand{\leq}{\leqslant}
\renewcommand{\geq}{\geqslant}

\newcommand{\I}{\mathcal{I} }
\newcommand{\B}{\mathcal{B} }

\newcommand{\ep}{\epsilon}

\newcommand{\Z}{\mathbb{Z} }

\renewcommand{\B}{\mathcal{B}}

\renewcommand{\L}{\mathcal{L}}
\newcommand{\G}{\mathcal{G}}

\newcommand{\E}{\mathcal{E}}

\renewcommand{\L}{\mathcal{L}}
\newcommand{\C}{\mathcal{C}}

\newcommand{\Fb}{\mathbf{F}}
\newcommand{\mF}{\mathbb{F}}

\begin{document}

%%
%% The title of the paper goes here.  Edit to your title.
%%

\title{Conway groupoids, regular two-graphs and supersimple designs}
 
%%
%% Now edit the following to give your name and address:
%% 

\author{Nick Gill}
\address{Department of Mathematics, University of South Wales, Treforest, CF37 1DL, U.K.}
\email{nicholas.gill@southwales.ac.uk}
\author{Neil I. Gillespie}
\address{Heilbronn Institute for Mathematical Research, Department of Mathematics, University of Bristol, U.K.}
\email{neil.gillespie@bristol.ac.uk}
\author{Cheryl E. Praeger}
\address{Centre for the Mathematics of Symmetry and Computation, University of Western Australia, Australia}
\email{cheryl.praeger@uwa.edu.au}

\author{Jason Semeraro}
\address{Heilbronn Institute for Mathematical Research, Department of Mathematics, University of Bristol, U.K.}
\email{js13525@bristol.ac.uk}

%%
%% If there are three of more authors they are added in the obvious
%% way. 
%%

%%%
%%% The following is for the abstract.  The abstract is optional and
%%% if not used just delete, or comment out, the following.
%%%

\begin{abstract}
A $2-(n,4,\lambda)$ design $(\Omega, \B)$ is said to be supersimple if distinct lines intersect in at most two points. From such a design, one can construct a certain subset of $\Sym(\Omega)$ called a ``Conway groupoid''. The construction generalizes Conway's construction of the groupoid $M_{13}$. It turns out that several infinite families of groupoids arise in this way, some associated with 3-transposition groups, which have two additional properties. Firstly the set 
of collinear point-triples forms a  regular two-graph, and secondly the symmetric difference of two intersecting lines is again a line. In this paper, we show each of these properties corresponds to a group-theoretic property on the groupoid and we classify the Conway groupoids and the supersimple designs for which both of these two additional properties hold. 
%We also generalize the notion of a Conway groupoid to the setting of $2-(n,3,2\lambda)$ designs. We present results in this setting dealing with the question of when the Conway groupoid is a subgroup of $\Sym(\Omega)$, and we show that the Higman-Sims sporadic simple group is an example of such a groupoid.
\end{abstract}



\maketitle 

\section{Introduction}
In his famous paper \cite{Co1}, John Conway used a ``game'' played on  the projective plane ${\mathbb P}_3$ of order $3$  to construct the sporadic Mathieu group $M_{12}$, as well as a special subset of $\Sym(13)$ which he called $M_{13}$, and which could be endowed with the structure of a groupoid.

In recent work (\cite{conwaygroupoids1, conwaygroupoids}), Conway's construction has been generalized to geometries other than ${\mathbb P}_3$, namely to supersimple $2-(n,4,\lambda)$ designs. In this more general context, the analogue of $M_{13}$ is a subset of $\Sym(n)$ that is known as a {\it Conway groupoid}. The aim of this paper is to classify an infinite family of Conway groupoids with the remarkable property that they are subgroups of $\Sym(n)$. They also have links to regular two-graphs and to 3-transposition groups. 

In order to state our main results, we briefly review the definition of a Conway groupoid (full definitions and background can be found in \S\ref{s: background}): we start with a $2-(n,4,\lambda)$ design $\De:=(\Omega,\B)$ in which any two lines intersect in at most two points. We call such designs {\it supersimple}, and note that distinct points $a, b$ of $\De$ lie in $\lambda$ lines $\{a,b,a_i,b_i\}$ such that the points $a_1,\dots, a_\lambda, b_1,\dots, b_\lambda, \in \Omega$ are pairwise distinct. We associate with $\{a,b\}$ the permutation 
\begin{equation}\label{eq:ss}
[a,b] = (a,b) \prod\limits_{i=1}^\lambda (a_i,b_i),
\end{equation} 
of $\Sym(\Omega)$, which we call an {\it elementary move}; we also set $[a,a]=1$, the identity permutation, for each $a\in\Omega$.  For an arbitrary sequence of points $a_0,a_1,\dots, a_k\in \Omega$, we extend this notation and define the permutation
\[
[a_0,a_1,a_2,\ldots,a_k]:=[a_0,a_1][a_1,a_2]\cdots[a_{k-1},a_k],
\]
called a {\it move sequence}. For a 
point $\infty\in \Omega$, the subset
\begin{equation}\label{eq:cg}
 \L_\infty(\De):= \{ [\infty,a_1,a_2,\ldots,a_k] \mid k \in \Z^+, \infty ,a_1,\dots, a_k \in \Omega \}\subset \Sym(\Omega)
\end{equation} 
is called the \textit{Conway groupoid} associated with the point $\infty$, and the \textit{hole-stabilizer} associated with $\infty$ is defined as 
\begin{equation}\label{eq:hs}
\pi_\infty(\De):=  \{ [\infty,a_1,a_2,\ldots,a_{k-1}, \infty] \mid k \in \Z^+, \infty ,a_1,\dots, a_{k-1} \in \Omega\}.
\end{equation} 
It is an easy matter to check that the set $\pi_\infty(\De)$ forms a subgroup of $\Sym(\Omega\setminus\{\infty\})$ and, moreover, its isomorphism class as a permutation group does not depend on the choice of the point $\infty$.

By way of example, if we consider, as Conway did, the design of points and lines of the projective plane ${\mathbb P}_3$, then we obtain $ \L_\infty({\mathbb P}_3)=M_{13}$ and $\pi_\infty({\mathbb P}_3)=M_{12}$. 
In the search for other interesting Conway groupoids, two particularly interesting phenomena have arisen: firstly, it turns out that the Conway groupoid $\L_\infty(\De)$ is sometimes not just a subset of $\Sym(\Omega)$, but a subgroup; secondly, by considering the set of collinear point-triples of $\De$, one can sometimes associate with $\De$ the structure of a {\it regular two-graph} (see Definition~\ref{d:cohdef}).\footnote{According to the first sentence of \cite{taylor}, ``regular two-graphs were introduced by G.~Higman in his Oxford lectures as a means of studying Conway's sporadic simple group $\cdot 3$ in its doubly transitive representation of degree $276$.''} %(each 4-subset of $\Omega$ contains  zero, two or four collinear triples). 
It turns out that these two properties, both separately and together, correspond to certain additional properties of the Conway groupoid, as our first main result makes clear:

\begin{thma}\label{t:l8gp}
Let $\De=(\Omega,\B)$ be a supersimple $2-(n,4,\lambda)$ with $n > 2\lambda+2$, and let $\C$ denote the set of all collinear triples of elements in $\Omega$.  Let $\infty \in \Omega$ and define $G:=\L_\infty(\De)$. Then the following hold.
\begin{itemize}
\item[(a)] If $G$ is a group then $G$ is primitive on $\Omega$.
\item[(b)] If $(\Omega,\C)$ is a regular two-graph then $\pi_\infty(\De)$ is transitive on $\Omega\setminus\{\infty\}$.
\item[(c)] If $(\Omega,\C)$ is a regular two-graph and $G$ is a group then $\pi_\infty(\De)$ is primitive on $\Omega\setminus\{\infty\}$.
\end{itemize}
\end{thma}

We remark that the condition $n>2\lambda+2$ is stated only for convenience. Any $2-(n,4,\lambda)$ design automatically satisfies $n\geq 2\lambda+2$; furthermore full information concerning the Conway groupoids corresponding to supersimple $2-(n,4,\lambda)$ designs with $n=2\lambda+2$ is given by Proposition~\ref{p: small lambda}.

Our other main results concern Conway groupoids satisfying the conditions in part (c), together with the following additional property on the supersimple design $\De=(\Omega, \B)$: 

\begin{equation}\label{e:symdiff}
\textrm{if }B_1,B_2 \in \B \textrm{ such that } |B_1 \cap B_2|=2, \textrm{ then } B_1 \triangle B_2 \in \B \tag{$\triangle$}
\end{equation}
where $B_1\triangle B_2$ denotes the {\it symmetric difference} of the lines $B_1$ and $B_2$. (Observe that the condition $ |B_1 \cap B_2|=2$ implies  that $|B_1 \triangle B_2|=4$.) Our next result gives a classification of such groupoids. Its statement refers to elementary moves defined in \eqref{eq:ss}, and also mentions 3-transposition groups, which are defined in Definition \ref{d:3trans}.

\begin{thmb}
Let $\De=(\Omega,\B)$ be a supersimple $2-(n,4,\lambda)$ design that satisfies  $\eqref{e:symdiff}$,
and let $\E$ be the set of elementary moves on $\De$. 
Let $\C$ be the set of collinear triples in $\B$ and suppose that $(\Omega, \C)$ is a regular two-graph.   
Then, for $\infty \in \Omega$, $(\L_\infty(\De), \E)$ is a $3$-transposition group and, for some positive integer $m$, one of the following holds:
\begin{itemize}
\item[(a)] $n=2^m$ and $\L_\infty(\De) \cong (\mathbb{F}_2)^m$;
\item[(b)] $n=2^{m-1}(2^m \pm 1)$ and $\L_\infty(\De) \cong \Sp_{2m}(2)$;
\item[(c)] $n=2^{2m}$ and $\L_\infty(\De) \cong 2^{2m}.\Sp_{2m}(2)$.
\end{itemize}
\end{thmb}

Our proof of Theorem B is independent of the Classification of Finite Simple Groups (CFSG); let us briefly explain our approach: we remarked above that when $n=2\lambda+2$, Proposition~\ref{p: small lambda} gives full information and, in particular, it implies that Theorem B, part (a) holds. In addition, one can prove fairly easily that the assumptions of Theorem B imply that $\L_\infty(\De)$ is a group (Lemma~\ref{l:lgrp}). For the situation where $n>2\lambda+2$ we now apply Theorem A, part (c) to conclude that $\L_\infty(\De)$ is 2-primitive on $\Omega$. 

At this point, if were we happy to use CFSG, we could invoke Taylor's classification~\cite{taylor2} of 2-transitive regular two-graphs. However we prefer to avoid reliance on the simple group classification, and instead we apply  Fischer's classification of finite 3-transposition groups; this is done in \S\ref{s: thmb}.

%\footnote{It appears that a different proof of Theorem B might be obtained by applying the theory of \textit{polar spaces}. More precisely,  if one fixes a point $\infty \in \Omega$, it can be shown that the suppositions of Theorem B imply that $(\Omega\backslash\{\infty\}, \G_{\De,\infty})$ is a polar space in the sense of Buekenhout-Shult. (Here $\G_{\De,\infty}$ is the set of all triples of points in $\Omega \backslash \{\infty\}$ which occur in a line with $\infty$.) This gives us strong information about the design $\De$, and we expect to be able to use this information to recover the design directly. These observations will form the basis of a separate paper.}
%

It is natural to ask at this point whether all of the possibilities for $\L_\infty(\De)$ that are listed in Theorem B can occur. The answer to this question is ``yes'' and we now present three families of designs that demostrate this. These families will also be of central importance in our final major result, Theorem C, below.

\begin{Ex}\label{exboolean}
 The \textit{Boolean quadruple system of order $2^m$}, where $m\geq2$, is the design $\De^b=(\Omega^b,\B^b)$ such that $\Omega^b$ is identified with the set of vectors in $\mathbb{F}_2^m$, and 
 \[
 \B^b:=\{\{v_1,v_2,v_3,v_4\} \mid v_i \in \Omega^b \mbox{ and } \sum_{i=1}^4 v_i = \textbf{0}\}.
 \]
Equivalently, we can define 
\[
\B^b=\{v+W \mid v\in \Omega^b, W\leq \mathbb{F}_2^m, \dim(W)=2\};
\]
that is, $\B^b$ is the set of all affine subplanes of $\Omega^b$.
It is easy to see that $\De$ is both a $3$-$(2^m,4,1)$ Steiner quadruple system and a supersimple $2$-$(2^m,4,2^{m-1}-1)$ design. Moreover, $\De^b$ satisfies property \eqref{e:symdiff} (see Lemma~\ref{l: boolean properties}). In what follows, we will often make statements like ``$\De$ is a Boolean system'' to mean that $\De$ is a Boolean quadruple system of order $2^m$ for some integer $m\geq 2$.
\end{Ex}

To describe the other two families, we need the following set-up: Let $m \geq 2$ and $V:=(\mathbb{F}_2)^{2m}$ 
be a vector space equipped with the standard basis. Define 
\begin{equation}\label{e:form}
e:=\begin{pmatrix}0_m&I_m\\0_m&0_m\end{pmatrix}, \qquad f:=\begin{pmatrix}0_m&I_m\\I_m&0_m\end{pmatrix}=e+e^T,
\end{equation}
where $I_m$ and $0_m$ represent the $m \times m$ identity and zero matrices respectively. Write elements of $V$ as 
row vectors and define $\varphi(u,v)$ as the alternating bilinear form  $\varphi(u,v):=ufv^T$. Also, 
write $\theta(u):=ueu^T \in \mathbb{F}_2$, so that 
$$
\theta(u+v)+\theta(u)+\theta(v)=\varphi(u,v).
$$ 
(Note that the right hand side equals $u e v^T + v e u^T$ while the left hand side is $u(e+e^T)v^T$.) 
Finally, for each $v \in V$ define $\theta_v(u):=\theta(u)+\varphi(u,v)$, and note that $\theta_0=\theta$.

\begin{Ex}\label{ex2}
Let $\De^a=(\Omega^a,\B^a)$, where $\Omega^a:=V$ and 
$$
\B^a:=\{\{v_1,v_2,v_3,v_4\} \mid v_1,v_2,v_3,v_4 \in \Omega^a, \sum_{i=1}^4 v_i=\textbf{0}, \sum_{i=1}^4 \theta(v_i)=0 \}.
$$ 
By \cite[Theorem B]{conwaygroupoids}, $\L_\infty(\De^a) \cong 2^{2m}.\Sp_{2m}(2)$, while $\pi_\infty(\De^a) \cong \Sp_{2m}(2)$. Indeed, taking $\infty$ to be the zero vector in $V$, it turns out that  $\pi_\infty(\De^a)={\rm Isom(V,\varphi)}$, the isometry group 
 of the formed space $(V, \varphi)$.
 \end{Ex}

\begin{Ex}\label{ex1}
For $\ep \in \mathbb{F}_2$, let $\De^\ep=(\Omega^\ep,\B^\ep)$, where $\Omega^\ep:=\{\theta_v  \mid v \in V, \theta(v)=\ep\}$ and  
$$
\B^\ep:=\{\{\theta_{v_1},\theta_{v_2},\theta_{v_3},\theta_{v_4}\} \mid \theta_{v_1},\theta_{v_2},\theta_{v_3},\theta_{v_4} \in \Omega^\ep, \sum_{i=1}^4 v_i= \textbf{0} \}.
$$
By \cite[Theorem B]{conwaygroupoids}, $\L_\infty(\De^\ep) \cong \Sp_{2m}(2)$, the isometry group of $\varphi$, while 
$\pi_\infty(\De^\ep) \cong \rm{O}_{2m}^{\ep'}(2)$, where $\ep'=\pm$ and $\ep=(1-\ep'1)/2$ (as an integer in $\{0,1\})$. \end{Ex}

Each of the designs in  Examples~\ref{exboolean}--\ref{ex1} satisfies property \eqref{e:symdiff} and their collinear triples form the lines of a regular two-graph (see Lemmas~\ref{l: boolean properties} and \ref{l: symp properties}), thereby proving that the three outcomes listed in Theorem B really occur.
Our final main result asserts that, in fact, the examples just given are the only designs that satisfy the assumptions of Theorem B.

\begin{thmc}%[Theorem C]
Let $\De=(\Omega,\B)$ be a supersimple $2-(n,4,\lambda)$ design that satisfies $\eqref{e:symdiff}$. 
Let $\C$ be the set of collinear triples in $\B$ and suppose that $(\Omega, \C)$ is a regular 
two-graph. Then there is an integer $m$ such that $\De$ is a ${2-(f(m),4,f(m-1)-1)}$ design, where 
one of the following holds:
\begin{itemize}
\item[(a)] $f(m)=2^m$ and $\De=\De^b$ as in Example~$\ref{exboolean}$;
\item[(b)] $f(m)=2^{2m}$ and $\De=\De^a$ as in Example $\ref{ex2}$;
\item[(c)] $f(m)=2^{m-1}(2^m + \ep'.1)$ and  $\De=\De^\ep$ is isomorphic to one of the designs in Example $\ref{ex1}$.
\end{itemize}
\end{thmc}

Our proof of Theorem C is entirely independent of Theorem B; indeed the whole approach to the proof is different from that for Theorem B because we use the theory of \textit{polar spaces}. More precisely, we prove in Proposition~\ref{p:red} that, for any point $\infty \in \Omega$, the assumptions of Theorem B (along with the condition $n>2\lambda+2$) imply that $(\Omega\backslash\{\infty\}, \C_{\infty})$ is a polar space in the sense of Buekenhout--Shult. (Here $\C_\infty$ is the set of all triples of points in $\Omega \backslash \{\infty\}$ which occur in a line with $\infty$.) 
In fact, the polar space $(\Omega\backslash\{\infty\}, \C_\infty)$ has the extra property that all lines in the space contain exactly three points. Such polar spaces were characterized in a special case by Shult \cite{Sh} and then later, in full generality, by Seidel \cite{Seidel}.\footnote{Note that \cite{Seidel} is an internal university report; it can be accessed in the volume of Seidel's selected works \cite{Se}.} We use the result of Seidel in \S\ref{s: thmc} to give a fairly short proof of Theorem C; our presentation of Seidel's result in that section (Theorem~\ref{t: seidel}) is couched in graph-theoretic terminology.

It is natural to ask about the connection between Theorems B and C. Although the two proofs given in this paper are independent of one another, the two theorems are in fact equivalent. That Theorem C implies Theorem B is an easy consequence of \cite[Theorem B]{conwaygroupoids}; the reverse implication is slightly more difficult and is not presented here.

We have chosen to give two proofs because we believe that the different approaches (one algebraic, one geometric) shed complementary light on the set-up being studied here. What is more, while a proof of Theorem B that goes via Theorem C appears somewhat shorter, an approach which goes via Theorem A (and hence a \textit{group-theoretic} analysis of $\L_\infty(\De)$) is likely to be applicable in more general contexts (see Question \ref{p:prob2} below). %We include our proof as testament to that, and we hope that this approach can give some new insight into the classification of these designs. 

\subsection{Context and open problems} 

As we mentioned above, the study of Conway groupoids was inspired by Conway's construction of $M_{13}$ in \cite{Co1}. This was generalized in \cite{conwaygroupoids1} to the context of supersimple $2-(n,4,\lambda)$ designs, and a classification programme for such groupoids was initiated in that paper and continued in \cite{conwaygroupoids}. Theorems~B and C may be regarded as contributions to this programme.
With this classification problem in mind, several questions arise. 
%One naturally wonders what the next step towards a full classification of all Conway groupoids of the form $\L_\infty(\De)$ might be. With this in mind, we ask the following question.

\begin{Qu}\label{p: extension}
Can Theorems B and C be extended to cover the situation where \eqref{e:symdiff} does not hold?
\end{Qu}

Our current state of knowledge about Conway groupoids suggests that $M_{13}$ is particularly special. To see this, we define a Conway groupoid $\L_\infty(\De)$ associated to a supersimple design $\De$ to be \textit{exotic} if $\L_\infty(\De)$ is not a group and $\pi_\infty(\De)$ is primitive. Thus $M_{13}$ is exotic, and \cite[Theorem D]{conwaygroupoids} gives strong bounds on the possible parameters of a design $\De$ for which $\L_\infty(\De)$ is exotic.%The original Conway groupoid $M_{13}$, constructed by Conway, is currently the only known example of such a groupoid.

\begin{Qu}\label{p:prob1}
Is $M_{13}$ the only exotic Conway groupoid? 
\end{Qu}

More broadly, we ask whether structures other than supersimple designs could be used to construct Conway groupoids.

\begin{Qu}\label{p:prob2}
Are there alternative combinatorial structures which can be used to define interesting groupoids? 
\end{Qu}

In unpublished work, the authors address Question~\ref{p:prob2} as follows: we consider a more general framework under which to define Conway groupoids by replacing the supersimple design $\De$ with an abstract function $f: \Omega^{2} \rightarrow \Sym(\Omega)$ that sends an unordered pair of points $\{a,b\}$ to an involution $[a,b]$. Several interesting examples arise in this context, and under certain conditions on $f$ one can show that the structure of the resulting groupoid is tightly controlled.

\subsection{Structure of the paper}
In \S\ref{s: background} we present background definitions and results concerning blocks designs, two-graphs, Conway groupoids and 3-transposition groups. We also prove that the designs presented in Examples~\ref{exboolean}--\ref{ex1} satisfy the assumptions of Theorems~B and C.
In \S\ref{s: thma} we prove Theorem~A. The main result of \S\ref{s: thmc prelim} is Theorem~\ref{t:cgrp} which is a precursor to Theorem B; the proof of Theorem B is completed in \S\ref{s: thmb}. Finally, in \S\ref{s: thmc}, we prove Theorem C.

\subsection{Acknowledgments}
The authors would like to thank Dr Ben Fairbairn and Dr Justin McInroy for their helpful comments. The last author would especially like to thank the University of Western Australia for its hospitality and for helping to finance a three week visit in April 2015.

\section{Background}\label{s: background}

\subsection{Block designs and two-graphs}\label{ss:designs}  For positive integers $t, k, n, \lambda$ such that $t\leq k\leq n$, a   $t-(n,k,\lambda)$ design $(\Omega,\B)$ consists of a finite set $\Omega$ of size $n$, whose elements are called the points of the design, together with a finite multiset $\B$ of subsets 
of $\Omega$ each of size $k$ (called \textit{lines}), such that any subset of $\Omega$ of size $t$ is contained in exactly $\lambda$ lines. Recall:

\begin{Def}
A $2-(n,4,\lambda)$ design $(\Omega,\B)$ is \textit{supersimple} if any two lines intersect in at most two points.
\end{Def}

Note that a supersimple design is {\it simple}, that is, distinct liens correspond to distinct subsets of $\Omega$.

\begin{Def}\label{d:cohdef}
A $2-(n,3,\mu)$ design $(\Omega,\C)$ is a \textit{regular two-graph} if for any 4-subset $X$ of $\Omega$, either $0, 2$ or $4$ of the $3$-subsets of $X$ lie in $\C$. A subset $X$ of $\Omega$ is \textit{coherent} if every $3$-subset of $X$ lies in $\C$.
\end{Def}

We note that each point of a $2-(n,3,\mu)$ regular two-graph $(\Omega,\C)$ is contained in $\mu (n-1)/6$ triples in $\C$, hence the name `regular'.  The following result is \cite[Proposition 3.1] {taylor}.

\begin{Lem}\label{l:sizen} Let $(\Omega,\C)$ be a $2-(n,3,\mu)$ regular two-graph. Then there is a constant $s$ such that each element of $\C$
is contained in exactly $s$ coherent $4$-subsets  of $\Omega$. Moreover, $n=3\mu-2s$.
\end{Lem}

\begin{Cor}\label{c:sizen}
Let $\De=(\Omega,\B)$ be a supersimple $2-(n,4,\lambda)$ design such that $(\Omega, \C)$ is a regular two-graph, where $\C$ is the set 
of collinear triples of $\De$. Then $n=6\lambda-2s$ where $s$ is the number of coherent $4$-subsets of $\Omega$ containing a given element of $\C$. In particular, $n$ is even.
\end{Cor}
\begin{proof}
Observe that $(\Omega,\C)$ is a $2-(n,3,2\lambda)$ design, and hence the result follows from Lemma \ref{l:sizen}.
\end{proof}

\subsection{Moves and groupoids}

Let $\De$ be a supersimple $2-(n,4,\lambda)$ design. For distinct $a,b \in \Omega,$ recall the definition of the elementary move from $a$ to 
$b$ in \eqref{eq:ss}.  Note that the right hand side of \eqref{eq:ss} is independent of the ordering of the 
lines, since the `supersimple' condition means that $\{a_i,b_i\}$ is disjoint from $\{a_j,b_j\}$ for $i\ne j$.  Recall also the definition of a move sequence  
$[x_0,x_1,\ldots,x_k]:=[x_0,x_1] \cdot [x_1,x_2] \cdots [x_{k-1},x_k]$.
For a point $\infty\in\B$, we are interested in the following three subsets of $\Sym(\Omega)$:
\begin{enumerate}
\item $\L(\De)$, the set of all move sequences;
\item $\L_\infty(\De)$, the set of all move sequences starting at $\infty$, called the \textit{Conway groupoid} of $\De$; and
\item $\pi_\infty(\De)$, the set of all move sequences which start and end at $\infty$, called the \textit{hole-stabilizer} of $\De$.
\end{enumerate}
Similarly we define $\L_x(\De)$ and $\pi_x(\De)$ for arbitrary $x\in\Omega$. 
We remark that $\pi_\infty(\De)$ is a subgroup of $\Sym(n-1)$, and is permutationally isomorphic to $\pi_x(\De)$ for each $x$ (so 
we may refer to it as \textit{the} hole-stabilizer of $\De$). Similarly, the isomorphism type of $\L_\infty(\De)$ as a groupoid 
does not depend on the choice of $\infty$. See \cite[\textsection2.2]{conwaygroupoids} for more discussion. Finally, for distinct $x,y\in\Omega$, we write 
\begin{equation}\label{d:overline}
\overline{x,y}:=\{z\in\Omega\,|\,\textnormal{there exists $\ell\in\B$ such that $x,y,z\in\ell$}\}
\end{equation}
and note in particular that $\overline{x,y}$ contains $x$ and $y$.
The next result is a simple observation and we omit the proof.

\begin{Lem}\label{l: kkk}
Let $\De$ be a supersimple $2-(n,4,\lambda)$ design. The following hold:
\begin{itemize}
\item[(i)]
$$
\L_\infty(\De)=\bigcup_{x \in \Omega} \pi_\infty(\De) \cdot [\infty,x],
$$
\item[(ii)] $\pi_\infty(\De)=\langle [\infty,x,y,\infty] \mid x,y \in \Omega \rangle.$
\end{itemize}
\end{Lem}

\begin{Lem}\label{l:lgrp}
Let $\De$ be a supersimple $2-(n,4,\lambda)$ design. Fix $\infty \in \Omega$ and define $G:=\L_\infty(\De)$.  The following are equivalent:
\begin{itemize}
\item[(i)] $G$ is a group;
\item[(ii)]  $G=\L(\De)=\langle [a,b] \mid a,b \in \Omega \rangle$;
\item[(iii)] $\L(\De)=\L_x(\De)$ for all $x \in \Omega$.
% \item[(iii)] For each $g \in \L(\De)$, $g \in \L_x(\De)$ for all $x \in \Omega$.
\end{itemize}
Furthermore, if one (and therefore all) of these conditions hold, then $G$ is transitive on $\Omega$  and $\stab_G(\infty)=\pi_\infty(\De)$.
\end{Lem}

\begin{proof} 
Write $H:=\langle [a,b] \mid a,b \in \Omega \rangle$ and note that $G \subseteq \L(\De) \subseteq H.$ 

{\it (i) implies (ii)}:\quad Note that   
$[a,b]=[\infty,a]\cdot[\infty,a,b]$, 
and so $[a,b]\in G$ since $G$ is a group. Hence $H \subseteq G$ and so $H=\L(\De)=G$ as required. 

{\it (ii) implies (iii)}: \quad Observe that, for each $x \in \Omega$, we have $\L_x(\De) \subseteq H=G$, and each move sequence $[x,x_2,\dots,x_k]=[x,\infty] [\infty,x,x_2,\dots,x_k]$, so 
$
|\L_x(\De)|=|[x,\infty] \cdot G|=|G|.
$ 
Thus $\L_x(\De)=G=\L(\De)$, and (iii) follows. 

{\it (iii) implies (i)}: \quad Let $g,h \in G$ and recall that $G=\L(\De)$; let $x$ be the last element of $\Omega$ in a move sequence 
corresponding to $g$. Now by (iii) there exist $y_i\in \Omega$ such that $h=[x,y_1,y_2 \ldots,y_l]$. Hence $g \cdot h \in G$ and $G$ is closed under composition.   
Since $G$ is finite this implies that $G$ is a group. 

To prove the final statement, suppose that $G$ is a group. Now $\pi_\infty(\De)$ clearly fixes $\infty$ and so is a subgroup of $\stab_G(\infty)$. Thus the 
length of the $G$-orbit containing $\infty$, namely $|G:\stab_G(\infty)|$, divides the index $|G:\pi_\infty(\De)|$. 
However, by Lemma~\ref{l: kkk}~(i), the index of $\pi_\infty(\De)$ in $G$ is equal to $|\Omega|$. It follows that  $G$ is transitive
and $\stab_G(\infty)=\pi_\infty(\De)$, as required.
\end{proof}

\begin{Lem}\label{l:lgrp2}
For a supersimple $2-(n,4,\lambda)$ design $\De$, the following are equivalent:
\begin{itemize}
\item[(i)] for all $a,b,c,d \in \Omega$, $[a,b]^{[c,d]} = [a^{[c,d]},b^{[c,d]}]$;
\item[(ii)] for all $a,b,c \in \Omega$, $[a,b]^{[b,c]} = [a^{[b,c]},c]$;
\item[(iii)] for all $a,b,c \in \Omega$, $[b,c]=[a,b,c,a^{[b,c]}]$.
\end{itemize}
In addition, if these conditions hold, then $\L_\infty(\De)$ (where $\infty \in \Omega$) is a group of automorphisms of $\De$.
\end{Lem}

\begin{proof}
First we show that conditions (ii) and (iii) are equivalent. Writing 
$[a,b]^{[b,c]} = [b,c][a,b][b,c]$, we see that condition (ii) for $a,b,c$ is equivalent to $[b,c] = [a,b][b,c][a^{[b,c]},c]$. On the other hand, by the definition of a move sequence, 
\[
[a,b,c,a^{[b,c]}]=[a,b][b,c][c,a^{[b,c]}]=[a,b][b,c][a^{[b,c]},c]
\]
and hence condition (iii) for $a,b,c$ is also equivalent to 
$[b,c] = [a,b][b,c][a^{[b,c]},c]$. Thus  conditions (ii) and (iii) are equivalent.

Next, if condition (i) holds, then taking $(a,b,c,d)$ in this condition as $(a,b,b,c)$ and noting that $b^{[b,c]}=c$, we obtain condition (ii) for $a,b,c$. Thus condition (i) implies condition (ii). 
Now assume that the equivalent conditions (ii) and (iii) hold, and let  
$a,b,c,d \in \Omega$.  Note that $[c,d]=[b,c,d,b^{[c,d]}]$. Then using this and several applications of condition (ii) yields 
$$
[a,b]^{[c,d]}=[a,b]^{[b,c][c,d][d,b^{[c,d]}]}=[a^{[c,d]},b^{[c,d]}].
$$

We now prove the final statement. We first show that $G:=\L_\infty(\De)$ is a group. To achieve this, we prove by induction on $k$ that each move sequence of length $k$ can be written as a move sequence starting from any given point $x$ of $\Omega$ and apply Lemma \ref{l:lgrp}. The identity element is equal to $[x,x]$ by convention, and it follows from (iii) that each elementary 
move $[a_1,a_2]$ may be represented by a move sequence starting with 
$x$. Thus the assertion is true for $k\leq2$. Suppose that $k>2$ and the assertion holds for all move sequences of length less than $k$, and consider $g=[x_1,x_2,\ldots,x_k]$ and a given point $x$. By induction, there exist $y_1,y_2 \ldots,y_l \in \Omega$ and $z_1,z_2 \ldots,z_m \in \Omega$ such that:
$$
[x_1,x_2]=[x,y_1,y_2 \ldots,y_l] \mbox{ and } [x_2,\ldots,x_k]=[y_l,z_1,\ldots,z_m].
$$ 
Composing these two move sequences yields the required expression for $g$.

Lastly, suppose that $\{a,b,c,d\}$ is a line of $\De$, and let $g\in G$.
Then $g$ is a move sequence, and condition (i) applied several times implies that $[a,b]^g=[a^g,b^g]$.  Then, since $(c,d)$ is a 2-cycle in the elementary move $[a,b]$, it follows that $(c^g,d^g)$ is a 2-cycle in the elementary move $[a^g,b^g]$. Therefore, $\{a^g,b^g,c^g,d^g\}$ is a line of $\De$. Thus $g\in \Aut(\De)$, as needed.
\end{proof}

\subsection{Examples}\label{s: examples}

In this section we prove that the three families of designs discussed in Examples~\ref{exboolean}, \ref{ex2} and \ref{ex1} satisfy the hypotheses of Theorem B and C. %At the end we will also briefly consider the groupoid $M_{13}$.

\begin{Lem}\label{l: boolean properties}
 Let $\De^b=(\Omega^b, \B^b)$ be a Boolean quadruple system of order $2^k$, and let $\C^b$ be the set of collinear triples of $\De^b$. Then $(\Omega^b, \C^b)$ is a regular two-graph and $\De^b$ satisfies \eqref{e:symdiff}. Furthermore, $(\L_\infty(\De^b),\pi_\infty(\De^b))=((\mathbb{F}_2)^m,\Id)$.
\end{Lem}
\begin{proof}
An immediate consequence of the definition of a Boolean quadruple system is that it is a $3-(2^k,4,1)$ design. Thus $\C^b$ contains all triples in $\Omega^b$ and $(\Omega^b, C^b)$ is the complete two-graph (hence, in particular, it is regular).

To see that $\De^b$ satisfies \eqref{e:symdiff}, consider intersecting lines $\{v_1,v_2, v_3, v_4\}$ and $\{v_1,v_2, v_5,v_6\}$ in $\B^b$. By definition
\[
 v_1+v_2+v_3+v_4=v_1+v_2+v_5+v_6=\textbf{0}.
\]
This implies that
\[
 v_1+v_2=v_3+v_4=v_5+v_6
\]
and we conclude immediately that $v_3+v_4+v_5+v_6=\textbf{0}$. In other words $\{v_3,v_4,v_5,v_6\}\in \B^b$, so $\De^b$ satisfies  \eqref{e:symdiff}. The last assertion is an immediate consequence of the definition of a Boolean quadruple system.

\end{proof}

The fact that $(\Omega^a, \C^a)$ and $(\Omega^\varepsilon, \C^\varepsilon)$ are regular two-graphs goes back to Taylor \cite{taylor2}. We give a self-contained proof here for the convenience of the reader:

\begin{Lem}\label{l: symp properties}
Each design $\De$ in Example \ref{ex2} or \ref{ex1} satisfies $(\triangle)$ and the set $\C$ of collinear triples in $\De$ forms a regular two-graph. Furthermore, $(\L_\infty(\De^a),\pi_\infty(\De^a))=(2^{2m}.\Sp_{2m}(2),\Sp_{2m}(2))$ and $(\L_\infty(\De^\ep),\pi_\infty(\De^\ep))=(\Sp_{2m}(2),O_{2m}^{\ep'}(2))$.
\end{Lem}

The final sentence in the lemma requires the identification of $\Omega^a$ with a vector space $V$ over the field $\mF_2$, as described in Example~\ref{ex2}. Recall that $\varphi:V\times V \to \mF_2$ is a particular alternating form.

\begin{proof}
Define 
$
\rho: V\times V \times V \to \mF_2$ by $(a,b,c) \mapsto \varphi(a,b)+\varphi(a,c)+\varphi(b,c).                
$
Observe that the lines $\{v_1,v_2,v_3,v_4\}$ in $\De^a$ are precisely those $4$-subsets of $\Omega$ for which $\sum_{i=1}^4 v_i=\textbf{0}$ and $\rho(a,b,c)=0$ for one (and hence any) 3-subset $\{a,b,c\}$ of $\{v_1,v_2,v_3,v_4\}$.
Observe further that the lines in $\De^\epsilon$ are precisely the lines $\{v_1,v_2,v_3,v_4\}$ in $\De^a$ with the property that $\theta_0(v_i)=\epsilon$ for $1 \leq i \leq 4$. 
Since 
\[
 \rho(a,b,c)+\rho(a,b,d)+\rho(a,c,d)+\rho(b,c,d) = 0
\]
for any four points $a,b,c,d\in \Omega$,  we conclude immediately that an even number of the 3-subsets of $\{a,b,c,d\}$ lies in $\C$ and $\Omega$ is a regular two-graph.

To prove \eqref{e:symdiff}, consider intersecting lines $\{v_1,v_2, v_3, v_4\}$ and $\{v_1,v_2, v_5,v_6\}$ in $\B$. By definition
\begin{align}\label{eq:tl}
 v_1+v_2+v_3+v_4&=v_1+v_2+v_5+v_6=\textbf{0},
 \end{align}
so that 
\[
 v_3+v_4+v_5+v_6+=\textbf{0}.
\]
We check that $\rho(v_3,v_4,v_5)=0$ from which it follows that  
$\{v_3,v_4, v_5,v_6\}\in \B$.
\begin{align*}
\rho(v_3,v_4,v_5)&=\varphi(v_3,v_4)+\varphi(v_3,v_5)+\varphi(v_4,v_5)
&(\mbox{definition of $\rho$})\\
&=\varphi(v_3,v_4)+\varphi(v_3,v_5)+\varphi(v_1+v_2+v_3,v_5) 
&(\mbox{by \eqref{eq:tl}})\\
&=\varphi(v_3,v_4)+\varphi(v_1,v_5)+\varphi(v_2,v_5) 
&(\mbox{bilinearity of $\varphi$})\\
&=\varphi(v_3,v_4)+\varphi(v_1,v_2)
&(\mbox{since $\rho(v_1,v_2,v_5)=0$})\\
&=\varphi(v_3,v_1+v_2+v_3)+\varphi(v_1,v_2)
&(\mbox{by \eqref{eq:tl}})\\ 
&=\varphi(v_1,v_2)+\varphi(v_1,v_3)+\varphi(v_2,v_3)=0,
&(\mbox{since $\varphi(v_3,v_3)=\rho(v_1,v_2,v_3)=0$})
 \end{align*}
as needed. The last assertion is a consequence of \cite[Theorem B]{{conwaygroupoids}}
%w 

%{\bf **** COMPLETE THIS PROOF.}
\end{proof}

\subsection{3-transposition groups}

We will need a number of results concerning 3-transposition groups; we gather these together below.

\begin{Def}\label{d:3trans}
A {\it 3-transposition group} is a pair $(G,\E)$, where $G$ is a group, $\E$ is a set of involutions in $G$ and the following conditions hold: 
\begin{itemize}
\item[(i)] $G=\langle \E \rangle;$ and
\item[(ii)] $\E$ is a union of $G$-conjugacy classes of involutions;
\item[(iii)] for all $g,h \in \E$, $gh$ has order $1, 2$ or $3$.
\end{itemize}
Elements of the set $\E$ are called {\it $3$-transpositions}. The pair $(G,\E)$ is called a {\it finite $3$-transposition group} if the group $G$ is finite.
\end{Def}

The first result is virtually immediate; a proof is given at \cite[1.1.2]{fischer}.

\begin{Lem}\label{l:fischer} 
Let $(G,\E)$ be a $3$-transposition group.  If $N$ is a normal subgroup of $G$, then
either $N=G$, or  $(G/N, \E N/N)$ is a $3$-transposition group.
\end{Lem}

%The next result is \cite[7.3]{aschbacher}.

%\begin{Thm}\label{t:asch} 
%Let $(G,\E)$ be a finite $3$-transposition group.  If $\E=\E_1\cup \E_2$ is a $G$-invariant  partition of $\E$, then on setting $G_1:=\langle \E_1 \rangle$ and $G_2:=\langle \E_2 \rangle,$ we have $$[G_1,G_2]=1 \mbox{ and } G=G_1G_2.$$ 
%\end{Thm}

Recall that $O_p(G)$ denotes the largest normal $p$-subgroup of a finite group $G$. The following result is an immediate consequence of \cite[Theorem 9.3]{aschbacher} and \cite[Theorem 2.2]{cuypers}.

\begin{Prop}\label{p:results} 
Let $(G, \E)$ be a finite $3$-transposition group and suppose that $\E$ is a conjugacy class in $G$ and that $Z(G)=1$.  Then $O_p(G)=1$ for all primes $p>3$, and exactly one of the following holds:
\begin{itemize}
\item[(i)] $O_2(G)\neq 1$;
\item[(ii)] $O_3(G)\neq 1$; or
\item[(iii)] $Q(G):=\langle gh\,|\, g,h\in \E, |gh|=2\rangle$ is a nonabelian simple group.
\end{itemize}
\end{Prop}

Finally we will require the celebrated classification of finite 3-transposition groups due to Fischer \cite{fischer}. Note that this classification is not dependent on the classification of finite simple groups (indeed the reverse is true). There are several versions of this theorem, and we prefer to use a slightly simplified version covering the case with trivial soluble radical.

\begin{Thm}\label{t: fischer}
Let $(G,\E)$ be a finite 3-transposition group. Suppose that $\E$ is a conjugacy class in $G$, and that $G$ contains no non-trivial solvable normal subgroups. Then, up to conjugation in $\Aut(G)$, one of the following holds:
\begin{enumerate}
\item $m \geq 5$, $G\cong \Sym(m)$  and $\E$ is the class of transpositions;
\item $m \geq 2$, $G\cong {\rm O}_{2m}^\pm(2)$ and $\E$ is the class of transvections;
\item $m \geq 2$, $G\cong \Sp_{2m}(2)$ and $\E$ is the class of transvections;
\item $m \geq 4$, $G\cong \rm{PSU}_m(2)$ and $\E$ is the class of (projective images of) transvections;
\item $m \geq 2$, $G$ is a subgroup of $\rm{PO}^\pm_{2m}(3)$ containing $\rm{P\Omega}^{\pm}_{2m}(3)$ and $\E$ is a class of (projective images of) reflections;
\item $G$ is one of $Fi_{22}$, $Fi_{23}$, $Fi_{24}$, 
$\Omega_8^+(2)\,:\, \Sym(3)$, or ${\rm P}\Omega_8^+(3)\,:\,\Sym(3)$, and in each case $\E$ is a unique class of involutions in $G$.
\end{enumerate}
\end{Thm}

\begin{proof}
The assumptions of \cite[Theorem (1.1)]{cuypers} hold. 
Since $G=\langle\E\rangle$, and $G$ contains no non-trivial solvable normal subgroups (so in particular $Z(G)=1$), it follows that 
$G$ is an almost simple group containing $\E$
(in fact generated by $\E$), and so $G$ is one of the groups given in (1) -- (6). (Information about the groups generated by $\E$ may be found, for example, in \cite[\S11]{aschbacher}).
\end{proof}

\section{Properties of Conway Groupoids}\label{s: thma} 

In this section we prove Theorem A. Throughout the section we assume the hypotheses of Theorem A, namely: 
\begin{itemize}
 \item $(\Omega,\B)$ is a supersimple $2-(n,4,\lambda)$ design with $n > 2\lambda+2$;
\item   $\infty$ is an element of $\Omega$, and $G:=\L_\infty(\De)$.
\end{itemize}

We also write $\C$ for the set of all collinear triples in $(\Omega,\B)$. 

%\begin{Lem}\label{l:blocklem}
%Suppose that $G$ is a group and that it preserves a system of imprimitivity with $m$ blocks of size $k$. Let $\{x_1,x_2,x_3\} \in \C$ and for $1 \leq i \leq 3$ let $\Delta_i$ be the block of imprimitivity containing $x_i$. If $\Delta_i \cap \Delta_j = \emptyset$ for all $i \neq j$ then writing $x_4:=x_3^{[x_1,x_2]}$, the block of imprimitivity $\Delta_4$ containing $x_4$ is distinct from $\Delta_1$, $\Delta_2$ and $\Delta_3$.
%\end{Lem}

%\begin{proof}
%Assuming the lemma is false, then since $g:=[x_1,x_2]$ moves $x_3$, $\Delta_4 \in \{\Delta_1,\Delta_2\}$. But $g$ interchanges $\Delta_1$ and $\Delta_2$, a contradiction. 
%\end{proof}

\begin{Prop}\label{p:lprim}
Part (a) of Theorem A holds: if $G$ is a group then $G$ is a primitive subgroup of $\Sym(\Omega)$.
\end{Prop}

\begin{proof} Suppose that  $G$ is a group. Then, by Lemma~\ref{l:lgrp}, $G$ is a transitive subgroup of $\Sym(\Omega)$. Suppose for a contradiction that $G$ 
 preserves a system of imprimitivity with $m$ blocks of size $k$, where $m\geq2, k\geq2$, and let $\Delta=\{\infty, a_2,\ldots,a_k\}$ be the block of imprimitivity that contains $\infty$.

\medskip
Since $n > 2\lambda+2$ there exists 
$y\notin\overline{\infty,a_2}$. Then $g:=[\infty,y]$ fixes $a_2$, so $g$ must fix $\Delta$ setwise, and hence 
$y=\infty^g \in \Delta$. It follows that every element in $(\Omega\setminus\overline{\infty,a_2})\cup\{\infty,a_2\}$ 
lies in $\Delta$, which implies that $k\geq n-2\lambda$. In particular the number of fixed points of $g$ is $n-2\lambda-2 < k$.

Now let $b \in \Omega \setminus \Delta$, so that the block of imprimitivity $\Delta_2$
containing $b$ is distinct from $\Delta$, and consider $h:=[\infty,b]$. Note that $h$ 
interchanges $\Delta$ and $\Delta_2$. On the other hand, since $n>2\lambda+2 = |\supp(h)|$, $h$ has a fixed point in $\Omega$, 
say $c$, and the block of imprimitivity $\Delta_3$ containing $c$ is fixed setwise by $h$ and hence is 
distinct from $\Delta$ and $\Delta_2$. Since $k=|\Delta_3|$ is larger than the number of fixed points of $h$, it follows that
$\supp(h)\cap\Delta_3$ contains a point, say $c'$, and since $\supp(h)=\overline{\infty,b}$, the set 
$\ell:= \{\infty,b,c',b'\}$ is a line, where $b':=c'^{[\infty,b]}$. Note that $b'$ lies in the  block $\Delta_3^{[\infty,b]}$ which is equal to $\Delta_3$.  Now consider the elementary move $h':=[\infty,b']$. Since $(\infty,b')$ is a 2-cycle of $h'$ the element $h'$ interchanges $\Delta$ and $\Delta_3$. However since also $(b,c')$ is a 2-cycle of $h'$ (since $\ell$ is a line containing $\infty,b'$) $h'$ should interchange $\Delta_2$ and $\Delta_3$. This contradiction completes the proof.
\end{proof}

\begin{Prop}\label{t:2trans}
Part (b) of Theorem A holds: if $(\Omega, \C)$ is a regular two-graph then $\pi_\infty(\De)$ is transitive on $\Omega\setminus\{\infty\}$.
\end{Prop}

\begin{proof}
Suppose that $(\Omega, \C)$ is a regular two-graph, and let  $a \in \Omega \backslash \{\infty\}$. We claim that
\begin{equation}\label{e:2trans1}
2\lambda+2 \leq |a^{\pi_\infty(\De)}|.
\end{equation}
Since $n > 2\lambda+2$, there exists $b \notin \overline{\infty,a}$. Then $a^{[\infty,a,b,\infty]}=
b \in a^{\pi_\infty(\De)}$.  It is sufficient to prove that $\overline{a,b} \subseteq a^{\pi_\infty(\De)},$ for then \eqref{e:2trans1} follows from $|\overline{a,b}|=2\lambda+2.$
To see this, let $c \in \overline{a,b}\backslash \{a,b\}$. Then, since $(\Omega,\C)$ is a regular two-graph, either $\infty \notin \overline{a,c}$ 
or $\infty \notin \overline{b,c}$ (but not both). Hence, either $[\infty,a,c,\infty]$ or $[\infty,b,c,\infty]$ 
fixes $\infty$ and maps $a$ to $c$, or $b$ to $c$, respectively, whence $c \in a^{\pi_\infty(\De)}$. Thus  \eqref{e:2trans1} is proved.

By \eqref{e:2trans1}, each orbit of $\pi_\infty(\De)$ in $\Omega\setminus\{\infty\}$ has length at least $2\lambda + 2$. Thus 
if  $n\leq 4\lambda+4$, there is not space for two orbits and so  $\pi_\infty(\De)$ is transitive on $\Omega\setminus\{\infty\}$.
We may therefore assume that $n > 4\lambda+4$. Let $a,b \in \Omega$ and observe that
$$
|\overline{a,\infty} \cup \overline{b,\infty}| \leq 2(2\lambda+2)-1=4\lambda+3
$$ 
so there exists $w \notin \overline{a,\infty}\cup \overline{b,\infty}$. Then $[\infty,a,w,\infty,w,b,\infty] \in \pi_\infty(\De)$ 
and sends $a$ to $b$. Thus  $\pi_\infty(\De)$ is transitive on $\Omega\setminus\{\infty\}$ in this case also.
\end{proof}

\begin{Prop}\label{p:2prim}
Part (c) of Theorem A  holds: if $\C$ is a regular two-graph and $G$ is a group then $\pi_\infty(\De)$ is primitive on $\Omega\setminus\{\infty\}$.
\end{Prop}

\begin{proof}
Suppose that $(\Omega, \C)$ is a regular two-graph and that $G$ is a group. By Propositions~\ref{p:lprim} and \ref{t:2trans},
$G$ is 2-transitive on $\Omega$, and it is sufficient to prove that $\pi_\infty(\De)$ is primitive on $\Omega\setminus\{\infty\}$. 
Suppose to the contrary that  $\pi_\infty(\De)$ acts imprimitively on $\Omega \setminus \{ \infty \}$ with $m$ blocks of size $k$,
where $m\geq2, k\geq2$, and $n-1=mk$. Let $a,b \in \Omega \setminus \{\infty\}$ lie in the same block of imprimitivity, say $\Delta$. 

Suppose first that $\infty \notin \overline{a,b}$. Choose $c \in \overline{a,b}$ so that there exists a line
 $\{a,b,c,d\} \in \B$ for some $d \in \Omega\setminus\{\infty\}$. We show first that $c \in \Delta$.
Consider the set of points $\{\infty,a,b,d\}$. Since  $(\Omega,\C)$ is a regular two-graph, exactly one of $a,b$ lies in $\overline{\infty,d}$. 
By interchanging the roles of $a$ and $b$ if necessary, we may assume that $a \in \overline{\infty,d}$ and $b \notin \overline{\infty,d}$. Then, considering 
the set of points $\{\infty,b,c,d\}$ we see that exactly one of $b,d$ lies in $\overline{\infty,c}$. 
If $d \notin \overline{\infty,c}$ then the permutation $g:=[\infty,a,d,\infty]$ fixes $a$ and sends $b$ to $c$. Thus 
$c=b^g \in \Delta^g=\Delta$ in this case. Assume now that  $b \notin \overline{\infty,c}$. Then considering 
$\{\infty,a,b,c\}$ we see that $\infty \in \overline{a,c}$. Hence the permutation $h:=[\infty,b,c,a,\infty]$ 
sends $a$ to $b$ (and hence fixes $\Delta$), and sends $b$ to $c$, so $c=b^h \in \Delta^h=\Delta$ in this case also.
Thus in either case $\Delta$ contains each point of $\overline{a,b}$, and so $k=|\Delta| \geq 2\lambda+2$. Now, by Corollary \ref{c:sizen}, $n$ is even so that $m=\frac{n-1}{k}$ is odd. Hence $m \geq 3$ so that by Corollary \ref{c:sizen} again,

$$
6\lambda < (2\lambda+2)m \leq km < n \leq 6\lambda,
$$ 
a contradiction.

Thus, $\infty\in\overline{a,b}$, and we note that this holds whenever $a,b$ lie in the same block of imprimitivity of $\pi_\infty(\De)$.  Since $n>2\lambda+2$, 
there exists $c\notin\overline{a,b}$. In particular $c\ne\infty$.
If $\infty \notin \overline{a,c}$ then by Lemma \ref{l:lgrp}, $[a,c]$ lies in  $\pi_\infty(\De)$, 
fixes $b$ and sends $a$ to $c$. Thus $c$ lies in the same block of imprimitivity $\Delta$ containing $a, b$. In particular $a,c$ lie in the same block of imprimitivity so $\infty \in \overline{a,c}$, which is a contradiction.  Hence $\infty \in \overline{a,c}$ and an identical argument (with the roles 
of $a$ and $b$ interchanged) shows that  $\infty \in \overline{b,c}$. This proves that exactly 
three 3-subsets of $\{\infty,a,b,c\}$ lie in $\C$, contradicting the fact that $\C$ is a regular two-graph. 
\end{proof}

Theorem A now follows from Propositions \ref{p:lprim}, \ref{t:2trans}, and \ref{p:2prim}.

\section{Towards a proof of Theorem B}\label{s: thmc prelim}

Recall condition \eqref{e:symdiff} defined in Section \S1 for a design $\De=(\Omega, \B)$:
\begin{equation*}
\textrm{If }B_1,B_2 \in \B \textrm{ such that } |B_1 \cap B_2|=2, \textrm{ then } B_1 \triangle B_2 \in \B. \tag{$\triangle$}
\end{equation*}

Throughout this section we assume the following hypotheses and notation. 

\begin{Hyp}\label{hyp4.1}
\mbox{}
\begin{itemize}
 \item $\De=(\Omega,\B)$ is a supersimple $2-(n,4,\lambda)$ design that satisfies \eqref{e:symdiff};
\item  $(\Omega, \C)$ is a regular two-graph, where $\C$ is the set of collinear triples of $\De$;
\item  $\infty \in \Omega$, and $G:=\L_\infty(\De)$; 
\item  $\E:=\{[a,b] \mid a,b \in \Omega\}$ is the set of elementary moves on $\De$.
\end{itemize}
\end{Hyp}
The main result of the section is  Theorem~\ref{t:cgrp}. One of its conclusions is that $G$ is a group. 

\begin{Thm}\label{t:cgrp}
If Hypotheses~$\ref{hyp4.1}$ hold, then $G$ is a subgroup of ${\rm Aut(\De)}$, %$G$ is $2$-primitive on $\Omega$, 
%$\E$ is a $G$-conjugacy class, 
and $(G, \E)$ is a 3-transposition group.  
\end{Thm}

\begin{Lem}\label{l:xrep}
If $\{w,x,y,z\} \in \B$ is a line then $[w,x]=[y,z]$.
\end{Lem}

\begin{proof}
This follows from the fact (using  \eqref{e:symdiff}) that $\{w,x,a,b\}$ is a line containing $\{w,x\}$ if and only if $\{y,z,a,b\}$ is a line containing $\{y,z\}$. 
\end{proof}

\begin{Lem}\label{l:sympeq}
For pairwise distinct $x,y,z \in \Omega$,
%\begin{equation}\label{e:conj}
\[
[y,z]\cdot [x,y] \cdot [y,z]=\left\{\begin{array}{ll}
           \textrm{$[x,y]$}, & \textrm{ if $x \in \overline{y,z}$;} \\
           \textrm{$[x,z]$}, & \textrm{ otherwise.}
          \end{array}\right.
\]
%\end{equation}

\end{Lem}

\begin{proof}
%\textbf{JS: Simplify this argument - avoid use of $\rho$}

%For each distinct triple of points $x,y,z \in \Omega$, define:

% \[ \rho(x,y,z):=\left\{\begin{array}{ll}
%            \textrm{$0$}, & \textrm{ if $x \in \overline{y,z}$;} \\
%           \textrm{$1$}, & \textrm{ if $x \notin \overline{y,z}$.}
%          \end{array}\right.
%\]
%Thus for each distinct quadruple $\{x,y,z,w\} \subseteq \Omega$ of points we have: $$\rho(x,y,z)+\rho(x,y,w)+\rho(x,z,w)+\rho(y,z,w) \equiv 0 \mod 2.$$ 

Let $w=[y,z] \cdot [x,y] \cdot [y,z]$.
Note first that $w$ is conjugate to $[x,y]$, and hence is an involution. 
Thus it suffices to show directly that the image under $w$ of each point $a \in \Omega$ 
is the same as its image under $[x,y]$ or $[x,z]$, according to whether $x \in \overline{y,z}$ or $x \notin \overline{y,z}$, respectively. 
It is straightforward to check that this is true if $a \in \{x,y,z\}$, so let $a \in \Omega \setminus \{x,y,z\}$. 
We consider three cases, according to whether zero, two or four of the $3$-subsets of $X:=\{x,y,z,a\}$ are collinear. 

If no 3-subsets of $X$ are collinear then $x \notin \overline{y,z}$ and since $a$ is fixed by all of $[x,y]$, $[y,z]$, $[x,z]$, it is also fixed by $w$. If all 3-subsets of $X$ are collinear 
then  $x \in \overline{y,z}$, and we have lines $\{x,y,a,b\}$, $\{x,z,a,c\}$, $\{y,z,a,d\} \in \B$,  for some $b,c,d \in \Omega$. 
Moreover, by  \eqref{e:symdiff}, the following $4$-subsets are also lines:
\begin{align*}
\{x,y,a,b\}  \triangle &\{x,z,a,c\}=\{y,z,b,c\};\\ 
\{x,y,a,b\} \triangle & \{y,z,a,d\}=\{x,z,b,d\};\\
\{x,z,a,c\}  \triangle & \{y,z,a,d\}=\{x,y,c,d\}.
\end{align*}
It can now be checked that $w$ and $[x,y]$ both send $a$ to $b$. 

Lastly suppose that exactly two of the 3-subsets of $X$ are collinear. There are six possibilities for the collinear pairs, corresponding to the six rows of Table~\ref{t:sympeq}.  Suppose first that $\{x,y,z\}$ is collinear. If $\{a, x, y\}$ is also collinear then $\{x,y,a,b\}$ is a line, for some $b \in \Omega$. 
In this case, if $\{b, y,z\}$ is collinear then for some $c \in \Omega$, $\{y,z,b,c\}$ is a line and hence $\{x,y,a,b\} \triangle 
\{y,z,b,c\}=\{x,z,a,c\}$ is a line, whence $\{a,x,z\}$ is collinear, which is a contradiction. Thus $\{b, y,z\}$ is not collinear, and hence  $a^w=a^{[x,y]}=b$ proving the assertions of row 1.  Next, if 
 $\{a, y, z\}$ is collinear, then $\{y,z,a,b\}$ is a line, for some $b \in \Omega$, and a similar argument to the previous case shows that $\{b,x,y\}$ is not collinear. Thus $a^w=a^{[x,y]}=a$ and the assertions of row 2 hold. If  $\{a, x, z\}$ is collinear, then $a$ is fixed by both $[y,z]$ and $[x,y]$, and hence $a^w=a^{[x,y]}=a$ and the assertions of row 3 hold. This completes the proof that $w=[x,y]$ if $x\in\overline{y,z}$. 

Now suppose that $\{x,y,z\}$ is not collinear. Then one of rows 4,5 or 6 holds.  
For row 4, there exist $b,c \in \Omega$ such that $\{x,y,a,b\}$ and $\{x,z,a,c\}$  are lines. 
Property \eqref{e:symdiff} implies that  $\{x,y,a,b\} \triangle \{x,z,a,c\}=\{y,z,b,c\}$ is a line, from which it 
follows that $w$ and $[x,z]$ both send $a$ to $c$. 
Similarly, for row 5, there exist $b,c \in \Omega$ such that $\{x,y,a,b\}$ and $\{y,z,a,c\}$  are lines. 
Property \eqref{e:symdiff} implies that  $\{x,y,a,b\} \triangle \{y,z,a,c\}=\{x,z,b,c\}$ is a line. If we also had a line $\{x,y,c,d\}$ for some $d\in\Omega$, then $\{y,z,a,c\} \triangle \{x,y,c,d\}=\{x,z,a,d\}$ is a line, which is a contradiction since $\{x,z,a\}$ is not collinear. Hence $\{x,y,c\}$ is not collinear, and so $w$ and $[x,z]$ both fix $a$. 
Finally for row 6,  there exist $b,c \in \Omega$ such that $\{y,z,a,b\}$ and $\{x,z,a,c\}$  are lines. 
Property \eqref{e:symdiff} implies that  $\{y,z,a,b\} \triangle \{x,z,a,c\}=\{x,y,b,c\}$ is a line. If we also had a line $\{y,z,c,d\}$ for some $d\in\Omega$, then $\{x,z,a,c\} \triangle \{y,z,c,d\}=\{x,y,a,d\}$ is a line, which is a contradiction since $\{x,y,a\}$ is not collinear. Hence $\{y,z,c\}$ is not collinear, and so $w$ and $[x,z]$ both send $a$ to $c$.
 This completes the proof.
\end{proof}

\begin{center}
\begin{table}[!ht]
{\small
\begin{tabular}{c|cccc}
%\hrule
Collinear triples in $X$ & $a^w$ & $a^{[x,y]}$ & $a^{[x,z]}$ \\ 
\hline
$\{x,y,z\}$, $\{a,x,y\}$ & $a^{[x,y]}$ & $a^{[x,y]}$ & -- \\
$\{x,y,z\}$, $\{a,y,z\}$ & $a$         & $a$         & -- \\
$\{x,y,z\}$, $\{a,x,z\}$ & $a$         & $a$         & -- \\ \hline
$\{a, x,y\}$, $\{a,x,z\}$ & $a^{[x,z]}$ & -- & $a^{[x,z]}$ \\
$\{a, x,y\}$, $\{a,y,z\}$ & $a$         &  -- & $a$\\
$\{a, y,z\}$, $\{a,x,z\}$ & $a^{[x,z]}$ & -- & $a^{[x,z]}$ \\ \hline
\end{tabular}
\caption{Comparing images for the Proof of Lemma~\ref{l:sympeq}: here  $w=[y,z] \cdot [x,y] \cdot [y,z]$.  }\label{t:sympeq}
}
\end{table}
\end{center}

Now we derive three more facts about the moves on $\De$.

%\begin{Cor} For all $a,b \in \Omega$ and $g\in G$, we have $[a,b]^g=[a^g,b^g].$ 
%\end{Cor}

%\begin{proof} %Let $a,b\in\Omega$, and $\{\{a,b,a_i,b_i\}\,|\,1\leq i\leq \lambda\}$ be the set of $\lambda$ lines that contain $a,b$.
%Then, by Lemma \ref{l:autgrp}, $\{\{a^g,b^g,a^g_i,b^g_i\}\,|\,1\leq i\leq \lambda\}$ is the set of $\lambda$ lines that contain $a^g,b^g$.  Thus
%$$[a^g,b^g]=(a^g,b^g)\Pi_{i=1}^{\lambda}(a_i^g,b_i^g)=[a,b]^g.$$ 
%\end{proof}

\begin{Lem}\label{l:braid}
Let  $x,y,z \in \Omega$ be pairwise distinct. 
\begin{enumerate}
 \item[(a)] Then,
 \[ o([x,y] \cdot [y,z])=\left\{\begin{array}{ll}
           \textrm{$2$}, & \textrm{ if $x \in \overline{y,z}$;} \\
           \textrm{$3$}, & \textrm{ if $x \notin \overline{y,z}$.}
          \end{array}\right.
\]
\item[(b)] if $x \notin \overline{y,z}$ then $[z,x,y,z]=[x,y].$
\item[(c)] $\L(\De)=\L_\infty(\De)$ is a group of automorphisms of $\De$.

\end{enumerate}
\end{Lem}

\begin{proof} (a) 
Let $X:=[x,y]$ and $Y:=[y,z]$. Note that, since $x,y,z$ are distinct, $X\ne Y$ and so the product $XY$ is not the identity. 
By Lemma \ref{l:sympeq}, either $XY=YX$ (if $x \in \overline{y,z}$) and $XY$ has order $2$, or $X$ and $Y$ satisfy the braid 
relation $XYX=YXY$, and $XY$ has order 3.

(b) Two applications of Lemma \ref{l:sympeq} yield 
$$
[z,x,y,z]=[z,x][x,y][y,z]=[z,x][x,y][z,x][z,x][y,z] $$ $$=[z,y][z,x][y,z]=[x,y],
$$ 
as required.

(c) We check that $\De$ satisfies the hypotheses of Lemma \ref{l:lgrp2}(iii). Let $x,y,z \in \Omega$. If $x \notin \overline{y,z}$ then $[x,y]=[z,x,y,z]=[z,x,y,z^{[x,y]}]$ by part (b). If $x \in \overline{y,z}$, so that $\{x,y,z,w\}$ is a line say, then $$[z,x,y,z^{[x,y]}]=[z,x]\cdot [x,y] \cdot [y,w]=[x,y] \cdot [z,x]\cdot [y,w]=[x,y] \cdot [z,x]\cdot [z,x]=[x,y],$$ where we use both the facts that $[z,x]=[y,w]$ (Lemma \ref{l:xrep}) and $[x,y]$ and $[x,z]$ commute (Lemma \ref{l:sympeq}). Thus the hypotheses of Lemma \ref{l:lgrp2}(iii) are satisfied and we conclude that $\L(\De)$ is a group of automorphisms of $\De$. In particular, $\L(\De)$ is a group, and so $\L(\De)=\L_\infty(\De)$ by Lemma \ref{l:lgrp}.

%We prove by induction on $k$ that each move sequence of length $k$ can be written as a 
%move sequence starting from any given point of $\Omega$. The identity element is equal to $[x,x]$ by convention, and it follows from part (b) and Lemma \ref{l:xrep} that each elementary 
%move $[z,w]$ may be represented by a move sequence starting with any given point $x$. Thus the assertion is true for $k\leq2$. Sppose that $k>2$ and the assertion holds for 
%all move sequences of length less than $k$, and consider $g=[x_1,x_2,\ldots,x_k]$ and a given point $x$. By induction, there exist $y_1,z_2 \ldots,y_l \in \Omega$ and $z_1,y_2 \ldots,z_m \in \Omega$ such that:
%$$
%[x_1,x_2]=[x,y_1,y_2 \ldots,y_l] \mbox{ and } [x_2,\ldots,x_k]=[y_l,z_1,\ldots,z_m].
%$$ 
%Composing these two move sequences yields the required expression for $g$.
\end{proof}

The previous two lemmas allow us to prove that the set $\E$ of elementary moves is closeD under
conjugation by elements of $\langle\E\rangle$.

\begin{Lem}\label{l:conj}
If Hypotheses~$\ref{hyp4.1}$ hold, then $[x,y]^g=[x^g,y^g]$ for all $g\in\langle\E\rangle$ and all $x,y\in\Omega$.
\end{Lem}

\begin{proof}
This is trivially true if $x=y$ since $[x,x]=1$ for all $x$, so we may assume that $x\ne y$. 
Moreover, it is sufficient to prove that $[x,y]^g=[x^g,y^g]$ for all $g\in\E$, so we may assume that
$g=[z,w]$ for distinct $z,w\in\Omega$. Furthermore, by Lemma  \ref{l:lgrp2} (the equivalence of conditions (i) and (ii)) we may assume that $x=w$. If also $y=z$ the condition is easy to check,
so assume that $y\ne z$, that is, $x$, $y$ and $z$ are pairwise distinct. Then Lemma~\ref{l:sympeq} applies. If $y\in\overline{x,z}$ then there is a line $\{x,y,z,c\}$
of $\De$;  Lemma~\ref{l:sympeq} gives  $[x,y]^g=[x,y]$, while $[x^g,y^g]=[z,c]$, which is equal to $[x,y]$ by Lemma~\ref{l:xrep}. 
Thus $[x,y]^g=[x^g,y^g]$ holds in this case.  Suppose now that $y\not\in\overline{x,z}$. Then by 
Lemma~\ref{l:sympeq}, $[x,y]^g=[z,y]$, while $[x^g,y^g]=[z,y]$, so again the condition holds.

%Thus we may assume that $x,y,w,z$ are pairwise distinct, and without loss of generality that $x\in\overline{z,w}$.  Thus there is a line $\{x,z,w,a\}$ in $\De$, and by Lemma~\ref{l:xrep}, $g=[z,w]=[x,a]$. If $y\not\in\overline{x,a}$, then by Lemma~\ref{l:sympeq}, $[x,y]^g=[x,y]^{[x,a]}=[a,y]$, and $[x^g,y^g]=[a,y]$, and the condition holds. Finally assume that $y\in\overline{x,a}$, so we have a line $\{x,y,a,b\}$ in $\De$ and by Lemma~\ref{l:xrep}, $g=[x,a]=[y,b]$. Then by Lemma~\ref{l:sympeq}, $[x,y]^g=[x,y]$, while $[x^g,y^g]=[a,b]$, which by Lemma~\ref{l:xrep} is equal to $[x,y]$, and again the condition holds.
\end{proof}

\begin{proof}[Proof of Theorem \ref{t:cgrp}]

%It follows from Lemma~\ref{l:braid}(c) that $G=\L_\infty(\De)$ is equal to $\L_x(\De)$ for all $x\in\Omega$. Hence by Lemma~\ref{l:lgrp}, $G$ is a group and $G=\langle\E\rangle$. Thus to prove that  $G\leq {\rm Aut(\De)}$ it is sufficient to prove, for a given line $\ell\in\B$, that 
%$\ell^g \in \B$ for each $g \in \E$. If $\supp(g) \cap \ell = \emptyset$, then $\ell^g=\ell$, so we may assume that $\supp(g) \cap \ell$ contains a point, say $a$. Let $b\in\ell\setminus\{a\}$.
%By Lemma~\ref{l:xrep}, we may write $g=[a,x]$ for some point $x\in\Omega\setminus\{a\}$. Let $\{a,b,a_i,b_i\}$ ($1\leq i\leq \lambda$) be the lines containing $\{a,b\}$, where $\ell=\ell_1$, so that 
%$[a,b]=(a,b)\prod_{i=1}^\lambda (a_i,b_i)$. Then (by the definition of permutation multiplication)
%\begin{equation}\label{eq:lgroup}
%[a,b]^g=(a^g,b^g)\prod_{i=1}^\lambda (a_i^g,b_i^g).
%\end{equation}
%Lemma \ref{l:sympeq} implies that $[a,b]^g=[a,b]$ if $x\in\overline{a,b}$, and  $[a,b]^g=[x,b]$ if $x\notin%\overline{a,b}$.
%In either case, $(a^g,b^g)$ and  $(a_1^g,b_1^g)$ are distinct 2-cycles in \eqref{eq:lgroup}, and it follows from property \eqref{e:symdiff} that
%their union $\{a^g,b^g,a_1^g,b_1^g\}=\ell^g$ is a line in $\B$ (and it contains $\{a,b\}$ or $\{x,b\}$ respectively).
%Thus we have proved that $G\leq {\rm Aut(\De)}$.
It follows from Lemma~\ref{l:braid}(c) that $G=\L_\infty(\De)$ is a subgroup of $\Aut(\De)$. It thus remains to prove that $(G,\E)$ is a 3-transposition group. To do this, we verify the three conditions (i),(ii) and (iii) of Definition \ref{d:3trans}. 
Condition (i), that $G=\langle\E\rangle$, follows from Lemma~\ref{l:braid}(c), and condition (ii) is proved in Lemma~\ref{l:conj}.  
To prove (iii), let  $x,y,z \in \Omega$. If these points are not pairwise distinct then $[x,y]\cdot[y,z]$ has order 1 or 2, so suppose they are pairwise distinct. We claim that
 \[ o([x,y] \cdot [y,z])=\left\{\begin{array}{ll}
           \textrm{$2$}, & \textrm{ if $x \in \overline{y,z}$;} \\
           \textrm{$3$}, & \textrm{ if $x \notin \overline{y,z}$.}
          \end{array}\right.
\]
To see this, let $X:=[x,y]$ and $Y:=[y,z]$. Note that, since $x,y,z$ are distinct, $X\ne Y$ and so the product $XY$ is not the identity. 
By Lemma \ref{l:sympeq}, either $XY=YX$ (if $x \in \overline{y,z}$) and $XY$ has order $2$, or $X$ and $Y$ satisfy the braid 
relation $XYX=YXY$, and $XY$ has order 3.
\end{proof}

If $n> 2\lambda +2$ then we can strengthen Theorem~\ref{t:cgrp} using Theorem A as follows.

\begin{Cor}\label{c:properties}
If Hypotheses~$\ref{hyp4.1}$ hold and, in addition, $n>2\lambda+2$, then the following hold:
\begin{itemize}
\item[(a)] $(G,\E)$ is a 3-transposition group;
\item[(b)] $G$ is a  2-primitive subgroup of $\Sym(\Omega)$;
\item[(c)] $Z(G)=1$;
%\item[(c)] $(G,\E)$ is a 3-transposition group. 
\item[(d)] $\E$ is a conjugacy class in $G$.
\end{itemize}
\end{Cor}
\begin{proof}
Part (a) is given by Theorem~\ref{t:cgrp}, and (b) follows from Theorem A~(c). 
Hence for any $a,b,c,d \in \Omega$ there exists $g \in G$ such that $a^g=c$ and $b^g=d$. 
Thus, by Lemma~\ref{l:conj}, $[a,b]^g=[a^g,b^g]=[c,d]$ and (d) is proved.
To prove (c), note that, since $n > 2\lambda+2$, for each $a \in \Omega$ there exists $b\in\Omega\setminus \overline{\infty,a}$. 
By Lemma \ref{l:sympeq}, the permutations $[\infty,a]$ and $[a,b]$ do not commute, so $G$ is nonabelian. 
Now $Z(G) \leq C_{\Sym(\Omega)}(G)=1$ by \cite[Theorem 4.2A]{permutation}.  
\end{proof}

\section{Proof of Theorem B}\label{s: thmb}

Before proceeding with the proof of  Theorem B, we record two background results from number theory.

Let ${x}$ and ${n}$ be positive integers. A \emph{primitive prime divisor} (sometimes called a \emph{Zsigmondy prime}) of ${x^n-1}$ is a prime divisor ${s}$ which
does not divide ${x^k-1}$ for any ${k<n}$. 
A special case of Zsigmondy's theorem \cite{zsig}, due to Bang \cite{bang}, tells us when primitive prime divisors exist with $x=2$.

\begin{Lem}\label{l:bang}
For positive integers $n$, $2^n-1$ has a primitive prime divisor, unless  $n=6$.
\end{Lem}

The other result that we need is elementary and can be found, for instance, as \cite[Lemma 6.3]{conwaygroupoids}.

\begin{Lem}\label{l:2a3b}
Suppose that $a, b, p$ are positive integers, and that $p^a \pm 1=2^b.$ Then either $a = 1$ or $p = 1$ or $(p,a)=(3,2).$
\end{Lem}

\subsection{The case \texorpdfstring{$n=2\lambda+2$}{n=2lambda+2}}

Let $\De=(\Omega,\B)$ be a supersimple $2-(n,4,\lambda)$ design. Since $\overline{a,b}$ contains $2\lambda+2$ distinct points, for any $a,b\in\Omega$, 
the number of points $n\geq 2\lambda+2$. The following result yields Theorems B and C in the case where equality holds. 
Recall that Boolean quadruple systems were defined in Example~\ref{exboolean}.

\begin{Prop}\label{p: small lambda}
Let $\De=(\Omega,\B)$ be a supersimple $2-(n,4,\lambda)$ design with $n=2\lambda+2$. The following conditions are equivalent.
\begin{enumerate}
\item $\De$ satisfies \eqref{e:symdiff};
\item $\De=\De^b$, a Boolean quadruple system of order $2^m$ for some integer $m\geq 2$;
\item $\pi_\infty(\De)=\{1\}$ and $\L_\infty(\De)$ is elementary-abelian of order $2^m$ for some integer $m\geq 2$.
\end{enumerate}
\end{Prop}
\begin{proof}
That (3) implies (2) is a consequence of \cite[Theorem B]{conwaygroupoids1}, and that (2) implies (3) follows from Lemma \ref{l: boolean properties}.
Next assume that condition (2) holds. Then condition \eqref{e:symdiff} is equivalent to the assertion that, if two affine subplanes of $\mathbb{F}_2^m$ intersect in a line, 
then their symmetric difference is an affine subplane. This is clearly true, so (1) holds.

Finally we must prove that (1) implies (2). Since $n=2\lambda+2$ and $\De$ is supersimple, $\De$ is a $3-(n,4,1)$ design. Define a commutative binary operation 
$*$ on $\Omega$ by setting $a*a:=\infty$ and $a*\infty=\infty *a=a$, for all $a\in \Omega$, and, for distinct $a,b \in \Omega\setminus\{\infty\}$, 
$$
a*b:=c\quad  \mbox{where $c$ is the unique point such that $\{\infty,a,b,c\}\in\B$.}
$$
 Now if for distinct $a,b,c \in \Omega \backslash \{\infty\}$ we have $\{\infty,a,b,d\}, \{\infty,b,c,e\}, \{\infty,a,e,x\} \in \B,$ then 
$$
(\{\infty,a,b,d\} \triangle \{\infty,b,c,e\}) \triangle  \{\infty,a,e,x\}=\{\infty,c,d,x\} \in \B
$$ 
and hence 
$$
a*(b*c)=a*e=x=d*c=(a*b)*c.
$$ 
It is easy to verify that $a*(b*c)=(a*b)*c$ also holds if $a,b,c$ are not pairwise distinct or if $\infty \in \{a,b,c\}$. Thus $*$ is associative, and it follows that 
$(\Omega,*)$ is an abelian group of exponent $2$, and hence is isomorphic to $(\mathbb{F}_2^m,+)$ for some integer $m \geq 2$. 
Furthermore, an identical argument to the one given above shows that $\{a,b,c,d\}$ is a line if and only if $a*b*c*d=\infty.$ Hence $\De=\De^b$ is a Boolean quadruple system, and (2) holds.
\end{proof}

It is worth noting that, since the Boolean quadruple system of order $2^m$ is a $3-(2^m,4,2^{m-1}-1)$, every 3-subset of $\Omega$ is collinear. Thus, if $\C$ is the set of all collinear 3-subsets of $\Omega$, then $(\Omega, \C)$ is a complete $3$-uniform hypergraph, and in particular is trivially a regular two-graph.

\subsection{Almost simple case}

In this subsection we assume that Hypotheses~$\ref{hyp4.1}$ hold, and in addition that $n>2\lambda +2$ and that $G$ contains no non-trivial solvable normal subgroup. (These are the hypotheses of Theorem B, plus $n\ne 2\lambda+2$, and an extra restriction on $G$.) 

%\begin{itemize}
% \item $\De=(\Omega,\B)$ is a supersimple $2-(n,4,\lambda)$ design that satisfies \eqref{e:symdiff};
%\item  $n>2\lambda +2$;
%\item  $(\Omega, \C)$ is a regular two-graph, where $\C$ is the set of collinear triples in $\De$;
%\item  $\infty \in \Omega$, and $G:=\L_\infty(\De)$ contains no non-trivial solvable normal subgroup; 
%\item  $\E:=\{[a,b] \mid a,b \in \Omega\}$ is the set of elementary moves on $\De$;
%\end{itemize}
%
By Corollary~\ref{c:properties}, $(G,\E)$ is a $3$-transposition group, $\E$ is a $G$-conjugacy class, $Z(G)=1$, and $G$ is $2$-primitive on $\Omega$. In particular the assumptions of Theorem~\ref{t: fischer} are satisfied, so  $G$ is an almost simple 
group and lies in one of the classes (1)-(6) listed there.  
Our job, then, is to prove that only class (3) can occur. We carefully avoid appeal to the finite simple group classification by 
using properties of the known almost simple 3-transposition groups.

\begin{Lem}\label{l:ecount}
$|\E|=\frac{n(n-1)}{2(\lambda+1)}$. 
\end{Lem}

\begin{proof}
Let $a$ and $b$ be distinct elements of $\Omega$. We are interested in those elements $[x,y]\in\E$ that contain the transposition $(a,b)$ in their cycle decomposition. The definition of an elementary move implies immediately that this will be the case if and only if $\{x,y\}=\{a,b\}$ or $\{a,b,x,y\}\in\B$. Thus there are $\lambda+1$ choices of $\{x,y\}$ for which $[x,y]$ contains $(a,b)$ in its cycle decomposition.

Now Lemma~\ref{l:xrep} implies that, for all of these choices, the resulting elementary moves are equal. The result follows by observing that there are $n(n-1)/2$ choices for the set $\{x,y\}$ in $\Omega$.
\end{proof}

We now consider case (1) of Theorem~\ref{t: fischer}. Note that, since $\Sym(6)\cong \Sp_4(2)$, the following lemma yields Theorem B in this case.

\begin{Lem}\label{l:alsym}
If $G \cong \Sym(m)$ with $m\geq5$, then $m=6$ and $n=10$.
\end{Lem}

\begin{proof}
By Corollary \ref{c:sizen}, we have $2(\lambda+1) < n < 6(\lambda+1)$, or equivalently, $2< \frac{n}{\lambda +1}< 6$. 
By Theorem~\ref{t: fischer}, $\E$ is the class of transpositions, so 
$|\E|=\binom{m}{2}$, and hence $m(m-1)=\frac{n(n-1)}{\lambda+1}$ by Lemma \ref{l:ecount}.  
Combining these two facts, we obtain $2(n-1)<m(m-1)<6(n-1)$, or equivalently,
$$
\frac{m(m-1)}{6} < n -1 < \frac{m(m-1)}{2}.
$$ 
In particular if $m=n$, then $m=5$, a contradiction. More generally, since $G$ is 2-transitive on $\Omega$, we apply \cite[Theorem 1]{bannai} and conclude that 
$
(m,n)\in\{(5,5), (5,6), (6,10)\}.
$
However if $m,n \in\{5, 6\}$, the equation in Lemma~\ref{l:ecount} implies that $\lambda\in\{0,\frac12\}$, which is 
a contradiction.
\end{proof}

\begin{Lem}\label{l:allie}
If $G$ is isomorphic to a group in one of the classes (2) -- (6) of Theorem~$\ref{t: fischer}$, then one of the following holds:
\begin{enumerate}
\item 
$G \cong \Sp_{2m}(2)$ and $n=2^{m-1} \cdot (2^m \pm 1)$, for some $m\geq2$, and the stabilizer of a point is isomorphic to ${\rm O}^{\pm}_{2m}(2)$;
\item $G\cong {\rm Fi}_{22}, {\rm Fi}_{23}$ or ${\rm Fi}_{24}$.
\end{enumerate}
\end{Lem}

\begin{proof}
Observe first that, if $G$  is isomorphic to a group in one of the classes (2) -- (6) of Theorem~\ref{t: fischer}, then $G$ is either a Fischer group (and is listed at part (2) above), or is an almost simple Chevalley group of normal or twisted type. Thus we assume that we are in this second case and we derive part (1).

We use the fact that, by assumption, $G$ admits a 2-transitive action. We apply \cite[Main Theorem]{cks}, which lists all almost simple Chevalley groups of normal or twisted type admitting a 2-transitive action. We obtain immediately that either part (1) above holds, or (possibly) one of the following occurs:
\begin{enumerate}
\item[(i)] $G\cong \rm{O}^-_4(2)\cong\Sym(5)$ with $n=5$ or $6$;
  \item[(ii)] $G\cong \rm{O}^+_6(2)\cong\Sym(8)$ with $n=8$ or $15$;
\item[(iii)] $\Alt(6) \cong \rm{P\Omega}^-_4(3) \leq G\leq \rm{PO}^-_4(3)$ with $n=6$ or $10$;
\item[(iv)] $\PSL_4(3) \cong \rm{PSO}^+_6(3) \leq G \leq \rm{PO}^+_6(3)$ with $n=40$.
\end{enumerate}
 In cases (i) and (ii), $G$ is a symmetric group and Lemma~\ref{l:alsym} immediately excludes them. In case (iii),  \cite[p. 46]{aschbacher} implies that $G\cong \Sym(6)\cong\Sp_4(2)$. Now Lemma~\ref{l:alsym} implies that part (1) holds.
 
In case (iv), $G$ has a socle isomorphic to $\PSL_4(3)$ and \cite{cks} implies that $G$ has a 2-transitive action if and only if $G\cong \PSL_4(3)$ or $\rm{PGL}_4(3)$. However, \cite{atlas} implies that $\rm{PO}^+_6(3)$ is a degree 2 extension of $\PSL_4(3)$ that is not isomorphic to $\rm{PGL}_4(3)$, and we conclude that $G$ has a 2-transitive action if and only if $G\cong\rm{PSO}^+_6(3)$. Now, by definition, $G$ is generated by (the projective image of) reflections which, in particular, have determinant equal to $-1$; thus $G$ is not a quotient of $\rm{SO}^+_6(3)$ and so cannot be isomorphic to $\rm{PSO}^+_6(3)$.
\end{proof}

Finally, we exclude the Fischer groups, by appealing to information about the minimal indices of their proper subgroups (we do not use the fact that these groups have no 2-transitive actions).

\begin{Lem}\label{l:alfi}
$G$ is not isomorphic to one of the groups $\Fi_{22}, \Fi_{23}$ or $\Fi_{24}$.
\end{Lem}

\begin{proof}
Referring to \cite{atlas, linton_wilson} and \cite[(37.5) and (37.6)]{aschbacher} we see that, for each of the groups $\Fi_{22}, \Fi_{23}$ and $\Fi_{24}$,  the minimum degree of a primitive permutation representation is equal to $|\E|$. Thus Lemma~\ref{l:ecount} implies that
\[
\frac{n(n-1)}{2(\lambda+1)}\leq n,
\]
 and so $n-1\leq 2(\lambda+1)$. On the other hand, by assumption, $2(\lambda+1)< n$ and we conclude that $n=2\lambda+3$. This contradicts the fact that $n$ is even (Corollary~\ref{c:sizen}).
\end{proof}

The following corollary draws together the results of this subsection (it also relies on Proposition~\ref{p: small lambda}).

\begin{Cor}\label{c: no solvable normal}
If Hypotheses~$\ref{hyp4.1}$ hold, $n>2\lambda+2$, and 
$G$ does not contain a non-trivial normal solvable subgroup, then for some $m \geq 2$, $G \cong \Sp_{2m}(2)$, $n=2^{m-1} \cdot (2^m \mp 1)$, and $\pi_\infty(\De) \cong {\rm O}^{\pm}_{2m}(2)$.
\end{Cor}

\subsection{Affine case}

In this subsection we assume that Hypotheses~$\ref{hyp4.1}$ hold, that $n>2\lambda+2$, and that $G$ contains a non-trivial solvable normal subgroup $N$.  (Here also, these are the hypotheses of Theorem B, plus $n\ne 2\lambda+2$, and an extra restriction on $G$.) 

Again, by Corollary~\ref{c:properties}, $(G,\E)$ is a $3$-transposition group, $\E$ is a $G$-conjugacy class, $Z(G)=1$, and $G$ is $2$-primitive on $\Omega$. Hence by \cite[Theorem  4.3B]{permutation}, $N$ is an elementary abelian $p$-group that acts regularly on $\Omega$.
Moreover Lemma~\ref{l:lgrp} implies that, for $\infty \in\Omega$, $\pi_\infty(\De)$ is the stabilizer of $\infty$ in the action of $G$ on $\Omega$. Hence we write $G_\infty:=\pi_\infty(\De)$ and observe that $G=N\rtimes G_\infty$. Observing that $N$ can be viewed as a vector space over a finite field of order $p$, we note that $G_\infty$ acts linearly on $N$ and so is isomorphic to a subgroup of $\textrm{GL}_d(p)$ where $|N|=p^d$. Furthermore, since $G$ is 2-transitive, $G_\infty$ acts transitively on the set of non-zero vectors in $N$.

\begin{Lem}\label{l:affsizen} 
$p=2$ and $n=2^d$ for some $d > 0$.
\end{Lem}
\begin{proof} This follows immediately from the fact that $n$ is even (Lemma \ref{l:sizen}), or the fact that $G_\infty$ is primitive on $\Omega\setminus\{\infty\}$.
\end{proof}

\begin{Prop}\label{p:L0} There is a conjugacy class $\E_\infty$ of $G_\infty$ such that ($G_\infty,\E_\infty)$ is a 3-transposition group. Furthermore, $Z(G_\infty)=1$.
\end{Prop}

\begin{proof} %We know that $G_\infty$ is a primitive subgroup of $\Sym(\Omega\backslash \{\infty\})$. Furthermore, since $n>2\lambda+2$ and $G$ is transitive, there exist $a, b \in \Omega$ such that $\infty \notin \overline{a,b}$. Now the primitivity of $G_\infty$ implies that $\langle \E \rangle \unlhd G_\infty$ acts transitively on  $\Omega\backslash \{\infty\}$. 

%Suppose first that $G_\infty$ is abelian. Then $\langle \E\rangle$ acts regularly on $\Omega\backslash\{\infty\}$. But, since $\langle \E\rangle$ has even order, this is a contradiction of Lemma~\ref{l:affsizen}.

%Thus $G_\infty$ is nonabelian and primitive and, what is more, by \cite[Thm. 4.2A]{permutation}, $Z(G_\infty) \leq C_{\Sym(\Omega\backslash\{\infty\})}(G_\infty)=1$.  

%{\bf{******The previous argument above seems to have been changed.  The current argument above seems incorrect to me. For example, $\langle \E \rangle$ is NOT a subgroup of $G_\infty$, because $\langle \E \rangle=G$. Moreover, if $\langle \E \rangle$ is supposed to be $\langle [a,b] \rangle$, we only know that $\langle [a,b] \rangle$ is a normal subgroup of $G_\infty$ if we assume that $G_\infty$ is abelian, so the argument seems out of order. I've written out the old argument, which I think is correct.***********}}

We know that $G_\infty$ is a primitive subgroup of $\Sym(\Omega\backslash \{\infty\})$. Furthermore, since $n>2\lambda+2$ and $G$ is transitive, there exist 
$a, b \in \Omega$ such that $\infty \notin \overline{a,b}$. Suppose first that $G_\infty$ is abelian, so $\langle [a,b] \rangle \unlhd G_\infty$. Then the 
primitivity of $G_\infty$ implies that $\langle [a,b] \rangle$ acts transitively (and therefore regularly) on $\Omega\backslash \{\infty\}$. 
But, since $\langle [a,b] \rangle$ has even order, this contradicts the fact that $n-1$ is odd (Lemma~\ref{l:affsizen}). Thus $G_\infty$ is nonabelian and primitive, and so, 
by \cite[Theorem 4.2A]{permutation}, $Z(G_\infty) \leq C_{\Sym(\Omega\backslash\{\infty\})}(G_\infty)=1$.

Since $N$ is a proper subgroup of $G$, $G/N\cong G_\infty$ is a $3$-transposition group 
with respect to $\mathcal{E}N/N$ by Lemma \ref{l:fischer}. Let $\E_\infty$ be the set of $3$-transpositions for $G_\infty$ obtained as the image of $\mathcal{E}N/N$ under an isomorphism from $G/N$ to $G_\infty$. Then ($G_\infty,\E_\infty)$ is a 3-transposition group. Moreover, since $\E$ is a $G$-conjugacy class (by Corollary~\ref{c:properties}), it follows that $\E\,N/N$ is a $(G/N)$-conjugacy class, and hence $\E_\infty$ is a $G_\infty$-conjugacy class.  
\end{proof}

\begin{Prop}\label{p:L1} 
$G_\infty$ is almost simple, and lies in one of the classes (1)-(6) of Theorem \ref{t: fischer}.
\end{Prop}

\begin{proof} By Proposition \ref{p:L0}, $G_\infty$ satisfies the hypotheses of Proposition \ref{p:results} and one of the conclusions (i), (ii) or (iii) holds. Suppose that (i) or (ii) holds. Since $G_\infty$ is primitive on $\Omega\backslash \{\infty\}$, any nontrivial normal subgroup acts transitively, and so $|\Omega\backslash \{\infty\}|=n-1 \in \{2^b,3^b\}$ for some $b > 0$. However $n=2^d$ 
(Lemma \ref{l:affsizen}), and hence $2^d-1=3^b$. This implies that $(d,b)=(2,1)$ by Lemma \ref{l:2a3b}, contradicting the fact that $n>2\lambda+2 \geq 4$. Thus Proposition \ref{p:results}~(iii) holds so $Q(G_\infty)=\langle gh \mid g,h \in \E, |gh|=2 \rangle$ is a nonabelian simple normal subgroup of $G_\infty$. Moreover since $|Q(G_\infty)|$ is even and $n-1$ is odd (by Lemma \ref{l:affsizen}), $Q(G_\infty)$ is not regular on $\Omega \backslash \{\infty\}$. Thus by \cite[Theorem 4.3B]{permutation}, $G_\infty$ is almost simple, and finally, by Theorem \ref{t: fischer}, $G_\infty$ lies in one of the classes (1)-(6) of that result.
\end{proof}

To complete the proof of Theorem B, it remains to consider the possibilities for $G_\infty$. We use the fact that $G_\infty$ acts transitively on $N^*$, the set of non-zero vectors of a $d$-dimensional vector space over $\mathbb{F}_2$. All transitive linear groups were classified by Hering -- with the key results appearing in \cite{hering}, and a complete list given in Liebeck \cite[Appendix]{liebeck}. We avoid using the classification from \cite{liebeck} since it depends on the finite simple group classification. Instead we use the fact that $G_\infty$ is in one of the six classes of Theorem~\ref{t: fischer}, together with some of the results in Hering's paper \cite{hering2} (which do not depend on the simple group classification). First we consider the groups from cases (1) and (3) of Theorem~\ref{t: fischer} to identify the examples in Theorem B~(c).

\begin{Lem}\label{l: ortho sym}
If $G_\infty \cong \Sym(m)$ for some $m \geq 5$ then $m=6$ and $d=4$, and Theorem B~(c) holds with $m=2$.
\end{Lem}

\begin{proof}
We apply \cite[Theorem 5.12]{hering} which (amongst other things) asserts that if a finite symmetric group acts transitively on the set of non-zero vectors of a $d$-dimensional vector space over $\mathbb{F}_2$, then $d=4$. Now the fact that $m=6$ follows by direct calculation, and since $\Sym(6) \cong \Sp_{4}(2)$, Theorem B~(c) holds.
\end{proof}

\begin{Lem}\label{l:sp2maff}
If $G_\infty \cong \Sp_{2m}(2)$ for some $m \geq 2$ then $d=2m$, and Theorem B~(c) holds.
\end{Lem}

\begin{proof}
Since $G_\infty$ is transitive on $N^*$, we must have $(2^d-1) \big| |G_\infty|$. If $d \neq 6$, let $r$ be a primitive prime divisor of $2^d-1$ (as in Lemma~\ref{l:bang}) and if $d=6$ let $r:=7$. In particular, note that $r$ divides 
$$
|G_\infty|=2^{m^2} \cdot \prod_{i=1}^m (2^{2i}-1).
$$
If $d > 2m$ then  $r \nmid 2^{2i}-1$ for  each $1 \leq i \leq m$, contradicting the fact that $r$ divides $|G_\infty|$. Hence $d \leq 2m$. 

We prove the reverse inequality. First suppose that $m \neq 3$, 
and let $r$ be a primitive prime divisor of $2^{2m}-1$, which 
exists by Lemma~\ref{l:bang}. Then $r$ divides $|G_\infty|$. Since $|G_\infty|$ divides
$$
|\GL_d(2)|=2^{\binom{d}{2}} \cdot \prod_{i=1}^d (2^i-1)
$$ 
an analogous argument shows that $2m \leq d$ and we conclude that 
$d=2m$, as required. 
If $m=3$, we simply check that $|\Sp_6(2)|$ does not divide the order of $|\GL_d(2)|$ for $d<6$ and so $d=6=2m$. Thus Theorem B~(c) holds for $m\geq3$. If $m=2$ then $G\cong\Sym(6)$ and Theorem B(c) holds by Lemma~\ref{l: ortho sym}.
\end{proof}

Next we eliminate the groups in case (2) and a related group from case (6).

\begin{Lem}\label{l: ortho affine}
The group $G_\infty$ is not isomorphic to  ${\rm O}_{2m}^\pm(2)$ for any $m \geq 2$, and $G_\infty \ne \Omega_8^+(2):\Sym(3)$.
% or a subgroup of $\Omega_8^-(2):\Sym(3)$ containing $\Omega_8^-(2)$.
\end{Lem}

\begin{proof}
Suppose that $G_\infty \cong  {\rm O}_{2m}^\pm(2)$ for some $m \geq 2$ and note that 
\[
|{\rm O}_{2m}^{\pm}(2)|=2 \cdot 2^{m(m-1)}(2^m\mp1)\prod_{i=1}^{m-1} (2^{2i}-1). 
\]
Exactly as for Lemma~\ref{l:sp2maff}, we use the fact that $2^d-1$ divides $|G_\infty|$ to conclude that $d\leq 2m$. Now \cite[Proposition 5.4.13]{kleidmanliebeck} asserts that the smallest module for $\rm{O}_{2m}^\pm(2)$ over a field of characteristic $2$ has dimension $2m$ and, moreover, it is the natural module. 
Since the sets of isotropic vectors and non-isotropic vectors are each invariant under the action of $\rm{O}_{2m}^\pm(2)$ on its natural module, and since each of these sets is non-empty, we conclude that $\rm{O}_{2m}^\pm(2)$ is not transitive on the set of non-zero vectors,
which is a contradiction. 

Now suppose that $G_\infty=\Omega_8^\pm(2)\,:\,\Sym(3)$. The arguments of the previous paragraph imply that $d=8$, so there are $|\Omega\setminus\{\infty\}|=2^d-1=255$ non-zero vectors. Also $G_\infty$ is transitive on $\Omega \backslash \{\infty\}$, and so its normal subgroup $\Omega^\pm_8(2)$ must have equal length orbits permuted transitively by $\Sym(3)$. However one of the $\Omega^\pm_8(2)$-orbits is the set of non-zero
singular vectors, of size 119, so this is not possible. 
\end{proof}

Now we consider case (4) of Theorem~\ref{t: fischer}. 

\begin{Lem}
$G_\infty$ is not isomorphic to $\rm{PSU}_m(2)$ for any $m \geq 4$.
\end{Lem}

\begin{proof}
Suppose that $G_\infty \cong \rm{PSU}_m(2)$ for some $m \geq 4$, and note that 
$$
|G_\infty|=\frac{2^{\binom{m}{2}}}{\delta} \cdot \prod_{i=2}^m (2^i-(-1)^i)
$$
where $\delta=\gcd(m,3)$. First we determine the value of $d$ for $m$ odd, and even, separately.

\medskip
{\it Claim: If $m$ is odd then $d=2m$.}\quad Suppose that $m$ is odd and 
let $r$ be a primitive prime divisor of $2^{2m}-1$ (which exists, by Lemma~\ref{l:bang} since $m\geq4$). Observe that, since $2^{2m}-1=(2^m-1)(2^m+1)$, $r$ divides $2^m+1$ and hence $r$ divides $|G_\infty|$. Then, as in the proof of Lemma \ref{l:sp2maff}, the fact that $|G_\infty|$ divides $|{\rm GL}_d(2)|$ implies that $2m \leq d$.
Conversely, if $d > 2m$ and $r$ is now a primitive prime divisor of $2^d-1$, then, as in the proof of Lemma \ref{l:sp2maff},  we have $r \nmid 2^{2i}-1$ for each $i \leq 2m$ and hence $r$ does not divide $|G_\infty|$, contradicting the fact that $2^d-1 \mid |G_\infty|$. Hence $d=2m$. 

\medskip
{\it Claim: If $m$ is even then $d=2m-2$.}\quad 
Suppose first that $m$ is even and $m>4$. Then $|G_\infty|$ is divisible by a primitive prime divisor of $2^{2m-2}-1$ and we conclude (as in the previous paragraph) that $2m-2 \leq d$.  For the reverse inequality, observe that a primitive prime divisor $r$ of $2^d-1$ (exists and) divides $|G_\infty|$ and hence $d \leq 2m-2$. So $d=2m-2$ in this case.
If $m=4$, then we one can check directly that $|\PSU_4(2)|$ does not divide $|\GL_d(2)|$ for $d\leq 5$ and hence $d\geq6$. If $d>6$, then a primitive prime divisor of $2^d-1$ exists and divides $|G_\infty|$ which in turn forces $d\leq 6$, which is a contradiction. Thus $d=2m-2=6$ in this case also.

Now we consult the list of small dimensional characteristic $2$ representations of $G_\infty$ in \cite[Proposition 5.4.11]{kleidmanliebeck}. When $m$ is even, the only $(2m-2)$-dimensional characteristic $2$ representation occurs when $m=4$. Similarly, when $m$ is odd, the only $2m$-dimensional characteristic 2 representation occurs when $m=5$. In both cases, one can check directly that $2^d-1$ does not divide the order of $G_\infty$, completing the proof.
\end{proof}

We next eliminate, with a single exception, the last infinite family, namely case (5), together with one further group from case (6).

\begin{Lem}
Suppose that $\ep \in \{+,-\}$ and $m\geq2$, and that either
$\rm{P\Omega}^{\ep}_{2m}(3)\leq G_\infty \leq \rm{PO}^{\ep}_{2m}(3)$,
or $G_\infty = \rm{P\Omega}^{+}_{8}(3)\,:\,\Sym(3)$. Then $m=2, \ep=-$, $d=4$ and $G_\infty\cong \Sym(6)\cong\Sp_4(2)$, 
as in Theorem B(c).
\end{Lem}

\begin{proof}
If $G_\infty = \rm{P\Omega}^{+}_{8}(3)\,:\,\Sym(3)$, then set $m=4$. 
First we use the results of \cite{landazuriseitz} as found, for instance, in \cite[Table 5.3.A]{kleidmanliebeck} to prove that $m<4$.
Suppose to the contrary that $m \geq 4$, and note that (for example, by \cite[Table 1 and Lemma 2.2]{praegerseressyalcinkaya}) $|\rm{O}^{\ep}_{2m}(3)|< 3^{2m^2-1}$. This is an upper bound for $|G_\infty|$ if $m>4$, while if $m=4$ then $G_\infty\leq \rm{P\Omega}^{+}_{8}(3)\,:\,\Sym(3)$ has order less than $3^{2m^2+1}=3^{33}$. On the other hand $2^{d} - 1$ divides $|G_\infty|$ and so $|G_\infty|> 2^{d-1}$. Moreover since $\rm{P\Omega}^{\ep}_{2m}(3)$ is contained in $\GL_d(2)$, we conclude 
from \cite[Table 5.3.A]{kleidmanliebeck} that $d=\dim(N) \geq 3^{2m-4}$.
From this we deduce that $2^{d-1}$ is greater than the upper bound for 
$|G_\infty|$, which is a contradiction. 

Thus $m\leq 3$. Indeed, by using the bounds in \cite[Table 5.3.A]{kleidmanliebeck} more carefully, we deduce that $m= 2$. 
Since $G_\infty$ is an almost simple group, by Proposition~\ref{p:L1}, we must have $\ep=-$, and so  $\rm{P\Omega}^-_{4}(3) \leq G_\infty \leq \rm{PO}^-_{4}(3)$.
Examining the list in \cite[P. 46]{aschbacher} we see that $G_\infty \cong \Sym(6)$ and the fact that $d=4$ follows from Lemma~\ref{l: ortho sym}. 
\end{proof}

Finally we show that the remaining groups in case (6) do not arise.

\begin{Lem}\label{l: fisch}
$G_\infty$ is not $\Fi_{22}, \Fi_{23}$ or $\Fi_{24}$.
\end{Lem}

\begin{proof}
Suppose that $G_\infty$ is one of these groups.
The dimensions $d$ of minimal $\mathbb{F}_2$-representations for $\Fi_{22}, \Fi_{23}$ and $\Fi_{24}$ are $78$, $782$ and $3774$ respectively \cite{jansen}. In each case, $2^d-1 > |G_\infty|$, a contradiction.
\end{proof}

Now, on combining Proposition \ref{p:L1} with Lemmas \ref{l: ortho sym} -- \ref{l: fisch} we obtain:

\begin{Cor}\label{c: solvable normal}
If Hypotheses~\ref{hyp4.1} hold, $n> 2\lambda +2$, and $G$ contains a non-trivial solvable normal subgroup, then $G \cong 2^{2m}.\Sp_{2m}(2)$, $n=2^{2m}$, and the stabilizer of a point is isomorphic to $\Sp_{2m}(2)$, for some $m\geq2$.
\end{Cor}

\noindent
{\it Proof of Theorem B.}\quad The proof follows immediately from Proposition \ref{p: small lambda}, and Corollaries \ref{c: no solvable normal} and \ref{c: solvable normal}.

\section{Proof of Theorem C}\label{s: thmc} 

Throughout this section we assume that Hypotheses~\ref{hyp4.1} hold.  In particular, the assumptions of Theorem C hold.  
\begin{comment}
As explained earlier, the number $n$ of points is at least $2\lambda+2$, and we first consider the case where $n=2\lambda+2$.

\begin{Lem}\label{l:B1} 
If $n= 2\lambda +2$, then Theorem C~(a) holds with $f(m)=2^m$.
\end{Lem}

\begin{proof}
It follows from Proposition~\ref{p: small lambda} that $\D=\D^b$ as in Example~\ref{exboolean}, so $n=2^m$ and $\lambda=2^{m-1}-1$ for some $m\geq2$. Thus the conclusion of Theorem C~(a) holds with $f(m)=2^m$.  
\end{proof}

From here on we assume that $n > 2\lambda +2$. 
\end{comment}

The proof we present in this case originated with the simple observation that, for the design $\De^a$ described in Example~\ref{ex2}, each maximal totally isotropic subspace of $\Omega^a=V$ coincides with (the set of points of) a maximal Boolean subdesign of $\De^a$. 
A similar property was seen to hold in $\De^+$ and this suggested that the theory of polar spaces may shed light on the geometry of designs satisfying Hypotheses~\ref{hyp4.1}. This turned out to be the case and led, eventually, to the proof that we now present.

In what follows we only need to consider polar spaces in which all lines are incident with exactly $3$ points. Such spaces were classified by Seidel~\cite{Seidel} (available on-line as a preprint, and also published in his `Selected works' \cite{Se}). We describe his result below using graph-theoretic language.
In that direction, we begin with some definitions: for $\infty\in\Omega$, we define $\G_{\De,\infty}=(V,E)$ as the graph with vertex set $V=\Omega\setminus\{\infty\}$, and edge set $E$ such that $\{a,b\}\in E$ if and only if $\{\infty, a, b\}\in\C$. This graph is called the \emph{derived graph} of the design $\De$.\footnote{We refer to $\G_{\De,\infty}$ as the derived graph {\it of the design} $\De$, but note that the definition of this graph refers to $\C$, rather than $\B$. Thus the definition could be extended to a more general setting including, in particular, all regular two graphs.}

\begin{Def}\label{d:tp} 
A graph $\G:=(V,E)$ satisfies the \emph{triangle property} 
if its edge set $E\ne\emptyset$ and, for each pair of adjacent vertices $u,v\in V$, there exists a vertex $w \in V$, adjacent to both $u$ and $v$, such that every vertex $x\in V \backslash \{u,v,w\}$ is adjacent to exactly one or exactly three vertices in the set $\{u,v, w\}$. We denote by $\mathcal{F}(u,v)$ the set of all vertices $w$ with this property, and if $|\mathcal{F}(u,v)|=1$, for all $u, v\in V$, then we say that $\G$ has the \emph{strong triangle property}. In this case we denote the unique vertex in $\mathcal{F}(u,v)$ by $f(u,v)$. 
\end{Def}

\begin{Lem}[{\cite[Lemma 4.2]{Seidel}}]\label{l:uniquef}
If a graph $\G$ has the triangle property and, further, if no vertex of $\G$ is adjacent to every other vertex, then $\G$ has the strong triangle property.
\end{Lem}

Our next result shows the relevance of the strong triangle property for us.\footnote{We asserted above that the totally isotropic subspaces of $\Omega^a=V$ coincide with (the set of points of) a maximal Boolean subdesign of $\De^a$. This observation easily implies that Proposition~\ref{p:red} holds for the designs $\De^a$; thus Proposition~\ref{p:red} can be thought of as a generalization of this observation.}

\begin{Prop}\label{p:red}
Suppose that Hypotheses~\ref{hyp4.1} hold and $ n > 2\lambda+2$. Let $\infty\in\Omega$. Then

\begin{itemize}
\item[(i)] $\G_{\De,\infty}$ has the strong triangle property; and
\item[(ii)] every line of $\D$ containing $\infty$ is of the form $\{\infty,a,b,f(a,b)\}$.
\end{itemize}
\end{Prop}

\begin{proof} 
First we prove the triangle property for $\G_{\De,\infty}$. Since $\infty$ lies in $2\lambda$ triples in $\C$, the edge set $E$ of  $\G_{\De,\infty}$ is non-empty.  Consider an edge $\{a,b\}\in E$, or equivalently  
$\{\infty,a,b\}\in\mathcal{C}$. Then there exists $c\in\Omega$ such that 
$\{\infty,a,b,c\}\in\B$, and therefore also
$$
\{\infty,a,c\},\{\infty,b,c\},\{a,b,c\}\in\mathcal{C}.
$$ 
Thus $c$ is adjacent to both $a$ and $b$ in $\G_{\De,\infty}$.   
Let  $x\in \Omega\setminus \{\infty, a,b,c\}=V \setminus \{a,b,c\}$,
and consider $\{a,b,c,x\}$.  Since $(\Omega,\C)$ is a regular two-graph and $\{a,b,c\}\in\mathcal{C}$, there are exactly one or three pairs $\{ r,s\}\subset\{a,b,c\}$ such that $\{r,s,x\}\in\C$.

\medskip\noindent
\emph{Claim. $\{ r,s,x\}\in\C$ if and only if $\{t,x\}\in E$, where $\{r,s,t\}=\{a,b,c\}$.}\quad  We prove this for the pair $\{a,b\}$, the proofs for the other pairs being identical. The triple $\{a,b,x\}\in\C$ if and only if there exists $d$ such that $\{a,b,x,d\}\in\B$. Using  property (\eqref{e:symdiff}) and the fact that $\{\infty,a,b,c\}\in\B$, we see that this holds if and only if  there exists $d$ such that$\{\infty,c,x,d\}\in\B$. The latter property is equivalent to the condition $\{\infty,c,x\}\in\C$, which in turn holds if and only if $\{c,x\}\in E$. This proves the claim. 

\medskip
Since there are exactly one or three pairs $\{ r,s\}\subset\{a,b,c\}$ such that $\{r,s,x\}\in\C$, it follows from the claim that $x$ is adjacent in $\G_{\De,\infty}$ to exactly one or three vertices in $\{a,b,c\}$. Thus $\G_{\De,\infty}$ has the triangle property.
Now since $n > 2\lambda+2$, for each vertex $v$ of $\G_{\De,\infty}$, there exists a vertex $u \notin \overline{\infty,v}$, that is, a vertex $u$ of $C_\infty$ which is not adjacent to $v$. Therefore, by 
Lemma~\ref{l:uniquef}, $\G_{\De,\infty}$ has the strong triangle property, and part (i) is proved.   

For part (ii), consider a line $B=\{\infty, a,b,c\}\in\B$ containing $\infty$. The arguments above show that the vertex $c$ has the property of Definition~\ref{d:tp} relative to $\{a,b\}$ and so $c\in\mathcal{F}(a,b)$. Since $\G_{\De,\infty}$ has the strong triangle property, this means that $c=f(a,b)$.   
\end{proof}

For the next part we need an alternative definition of the designs $\De^\ep$ from Example~\ref{ex1}: let $V$ be a $2m$-dimensional vector space over $\Fb_2$ and let $Q^{\ep'}:V\to \Fb_2$ be a non-degenerate quadratic form of type $\ep'$ (for $\ep'=\pm$) which polarizes to the 
non-degenerate alternating form $\varphi$ on $V$. Write $\ep=(1-\ep'1)/2$ and define
\begin{align*}
 \Omega^\ep &= \{ v\in V \mid Q^{\ep'}(v)=0\}; \\
 \B^\ep &= \{ \{v_1,v_2,v_3,v_4\} \mid v_1,v_2,v_3,v_4\in\Omega^\ep, \sum\limits_{i=1}^4 v_i=\textbf{0}\}.
\end{align*}
Now we define $\De^\ep=(\Omega^\ep, \B^\ep)$. The fact that this definition is consistent with the definition given in Example~\ref{ex1} is (almost) immediate for $\ep'=+$; for $\ep'=-$ it is an exercise.

We are now ready to state Seidel's classification result. We discussed it above in terms of polar spaces, although the statement we use concerns regular two-graphs whose derived graphs have the strong triangle property.

\begin{Thm}[Seidel's Classification Theorem]\label{t: seidel}
Suppose that a graph $\G=(V,E)$ satisfies the triangle property. Then $\G$ is one of the following:
\begin{enumerate}
 \item an edgeless graph, i.e. $E=\emptyset$;
 \item a complete graph;
 \item $\G_{\De^a, \textbf{0}}$, the derived graph of a design $\De^a$ at the vertex $\textbf{0}$;
 \item $\G_{\De^\ep, \textbf{0}}$, the derived graph of a design $\De^\ep$ at the vertex $\textbf{0}$, for some $\ep\in\mathbb{F}_2$.
\end{enumerate}
Conversely, all of the listed graphs satisfy the triangle property.
\end{Thm}

We are almost ready to derive Theorem C from Seidel's classification; we need one lemma first.

\begin{Lem}\label{l: reconstruct}
 Suppose that $\De_1$ and $\De_2$ are two designs satisfying Hypotheses~\ref{hyp4.1} with $n> 2\lambda+2$. Let $\infty_1$ (resp. $\infty_2$) be a point in $\De_1$ (resp. $\De_2$). If $\G_{\De_1, \infty_1}$ and $\G_{\De_2,\infty_2}$ are isomorphic as graphs, then $\De_1$ and $\De_2$ are isomorphic as designs.
\end{Lem}
\begin{proof}
Let $\De_i=(\Omega_i,\B_i)$, and $\G_i:= \G_{\De_i, \infty_i}$, for $i=1,2$.
Let $\phi: \G_1\rightarrow \G_2$ be a graph isomorphism, and extend $\phi$ to a bijection $\Omega_1\rightarrow\Omega_2$ by defining $\phi:\infty_1\mapsto \infty_2$. It is sufficient to show that the image under $\phi$ of each line in $\B_1$ is a line in $\B_2$.
We begin by considering the lines containing $\infty_i$. By Proposition~\ref{p:red}, the graphs $\G_i$ have the strong triangle property, and every line of $\De_i$ containing $\infty_i$ is of the form $\{\infty_i,a,b,f_i(a,b)\}$ for $a,b$ vertices of $\G_i$ (where we write $f_i$ for the function $f$ on $\G_i$). Let $\ell=\{\infty_1,a,b,f_1(a,b)\}\in\B_1$ and let $a'=\phi(a), b'=\phi(b)$. By the definition of $f_1(a,b)$, it follows that $\phi(f_1(a,b))=f_2(a',b')$, and hence $\phi(\ell)=\{\infty_2,a',b',f_2(a',b')\}$ is a line of $\De_2$.

Now consider a line $\ell:=\{a,b,c,d\}\in\B_1$ which does not contain $\infty_1$. Then $\{a,b,c\}$ is a collinear triple from $\De_1$, and applying the two-graph property to the 4-subset $\{\infty_1,a,b,c\}$, we see that $\infty_1$ is collinear with at least one of $\{a.b\}, \{b,c\}, \{a,c\}$. Without loss of generality we may assume that $\{\infty_1,a,b\}$ is collinear so we have a second line $\ell_1:=\{\infty_1,a,b,f_1(a,b)\}\in\B_1$. Moreover, by the symmetric difference property \eqref{e:symdiff}, $\ell_1':=\{\infty_1,f_1(a,b),c,d\}$ is also a line in $\B_1$, and so by the argument of the previous paragraph, we have $d=f_1(f_1(a,b),c)$. 
Let $a'=\phi(a), b'=\phi(b), c'=\phi(c)$ and $d'=\phi(d)$. 
Applying the argument of the previous paragraph again, we see that the images under $\phi$ of $\ell_1$ and $\ell_1'$ are lines of $\B_2$ and are $\{\infty_2, a', b', f_2(a',b')\}$ and $\{\infty_2, f_2(a',b'),c',d'\}$ respectively, with $d'=f_2(f_2(a',b'),c')$. Then, by the symmetric difference property \eqref{e:symdiff} for $\De_2$, the 4-subset $\{a',b',c',d'\}=\phi(\ell)$ is also a line of $\De_2$.   This completes the proof. 
\end{proof}

\begin{proof}[Proof of Theorem C]
We assume that Hypotheses~\ref{hyp4.1} hold. If $n=2\lambda+2$, then, since $\De$ is supersimple, we conclude that $\De$ is a $3-(n,4,1)$ design. Thus  $\G_{\De,\infty}$ is the complete graph and so satisfies the triangle property.
On the other hand, if $n > 2 \lambda + 2$, then, by Proposition~\ref{p:red},
$\G_{\De,\infty}$ satisfies the triangle property. Thus this conclusion holds in all cases. 

We now apply Seidel's Classification Theorem~\ref{t: seidel} to conclude that $\G_{\De,\infty}$ is of one of the five listed types. Case (1) can be discarded immediately as it would imply that $\De$ contained no lines. 

Case (2), on the other hand, implies that $\De$ is a $3-(n,4,1)$ design and, by counting as before, we conclude that $n=2\lambda+2$. Now Proposition~\ref{p: small lambda} implies that $\De=\De^b$, a Boolean quadruple system. 

Finally, for cases (3), (4) and (5), Lemma~\ref{l: reconstruct} implies that $\De$ is either $\De^a$ or $\De^\ep$.

%Lemma 4.9 of \cite{Seidel} now identifies three cases I, II, III, together with certain of their parameters. In particular $n=|\Omega|=f(m)$ is as in Example~\ref{ex1} (with $\ep=0$), Example~\ref{ex2}, or Example~\ref{ex1} (with $\ep=1$), respectively. Seidel then studies the eigenvalues of the adjacency matrix for $\G_{\De,\infty}$, and proves first, in \cite[Lemma 4.13]{Seidel} that, in Case II, $\G_{\De,\infty}$ is `the symplectic graph $S(2m,2)$', and his analysis of this graph and its associated regular two graph in \cite[Theorems 3.2 and 3.4]{Seidel} shows that $D=\D^a$ as in Example~\ref{ex2} and that $\lambda=2^{2m-2}-1=f(m-1)-1$, so Theorem C~(b) holds. Next he proves in \cite[Lemma 4.14]{Seidel} that in cases I and III, $\G_{\De,\infty}$ is `the orthogonal graph $O^+(2m,2)$ or $O^-(2m,2)$', respectively, and his analyses of these graphs and their associated regular two graphs in \cite[Theorems 3.8 and its proof]{Seidel} shows that $D=\D^\ep$ as in Example~\ref{ex1} and that $\lambda=2^{m-1}(2^m+\ep'.1)=f(m-1)-1$, with $\ep,\ep'$ as in Example~\ref{ex1}. Thus Theorem C~(c) holds in this case.       
\end{proof}

\bibliographystyle{amsalpha}
\bibliography{refs2}

\end{document}